\documentclass[11pt,reqno]{amsart}

\usepackage{amsmath}
\usepackage{amssymb}
\usepackage{amsthm}
\usepackage{amscd, amsfonts, mathrsfs}

\usepackage{cases}
\usepackage{xfrac}

\usepackage[dvips]{epsfig}
\usepackage{verbatim} 

\usepackage{epsf}
\usepackage[bookmarksnumbered,pdfpagelabels=true,plainpages=false,colorlinks=true,
            linkcolor=black,citecolor=black,urlcolor=blue]{hyperref}

\theoremstyle{plain}
\newtheorem{theorem}{Theorem}[section]
\newtheorem{lemma}[theorem]{Lemma}
\newtheorem{prop}[theorem]{Proposition}
\newtheorem{coro}[theorem]{Corollary}

\newtheorem{assump}[theorem]{Assumption}

\theoremstyle{definition}
\newtheorem{rema}[theorem]{Remark}
\newtheorem{notation}[theorem]{Notation}
\newtheorem{defi}[theorem]{Definition}

\numberwithin{equation}{section}

\newenvironment{ sys_eq }{\ left \ lbrace \ begin { array }{ @ {} l@ {}}}{\ end { array }\ right .}

\newcommand{\cA}{\mathcal A}
\newcommand{\cB}{\mathcal B}

\newcommand{\cD}{\mathcal D}
\newcommand{\cE}{\mathcal E}

\newcommand{\cH}{\mathcal H}

\newcommand{\cL}{\mathcal L}

\newcommand{\cP}{\mathcal P}

\newcommand{\cR}{\mathcal R}

\newcommand{\cT}{\mathcal T}

\newcommand{\ccE}{\mathscr{E}}

\newcommand{\al}{\alpha}
\newcommand{\be}{\beta}
\newcommand{\ga}{\gamma}

\newcommand{\de}{\delta}

\newcommand{\La}{\Lambda}

\newcommand{\om}{\omega}
\newcommand{\Om}{\Omega}

\newcommand{\RR}{\mathbb R}

\newcommand{\rar}{\rightarrow}

\newcommand{\ve}{\varepsilon}

\newcommand{\id}{\operatorname{id}}
\newcommand{\dive}{\operatorname{div}}

\newcommand{\p}{\parallel}
\newcommand{\n}{\parallel}

\newcommand{\half}{\sfrac{1}{2}}

\title[Large surface tension]{On the limit of Large surface tension for a fluid motion with free boundary}

\author[Disconzi]{Marcelo M. Disconzi}
\address{Department of Mathematics\\
Vanderbilt University, Nashville, TN 37240, USA}
\email{marcelo.disconzi@vanderbilt.edu}

\author[Ebin]{David G. Ebin}
\address{Department of Mathematics\\
Stony Brook University\\ Stony Brook, NY 11794-3651, USA}
\email{ebin@math.sunysb.edu}

\begin{document}

\begin{abstract}
We study the free boundary Euler equations in two spatial dimensions.
We prove that if the boundary has constant curvature, then solutions
of the free boundary fluid motion converge to solutions of the Euler equations 
in a fixed domain when the coefficient of surface tension tends to infinity.
\end{abstract}


\maketitle

\tableofcontents

Keywords: Euler equations, free boundary, surface tension.

MSC: 35L60, 35Q31, 35Q35.


\section{Introduction. \label{intro}}

Consider the initial value problem for the motion of an incompressible inviscid fluid 
with free boundary, whose equations of motions are 
given  by\footnote{Here $x$ is the Lagrangian coordinate of the fluid particle whose Euler coordinate at time $t$ is $\eta(t,x)$.}
\begin{subnumcases}{\label{free_boundary_full}}
 \ddot{\eta}  = - \nabla p \circ \eta   & in $ \Om$, \label{basic_fluid_motion_full} \\
  \operatorname{div} (u) = 0  & in  $\eta(\Om)$, \label{equation_p_full} \\
 \left. p \right|_{\partial \eta(\Om)} = k \cA   & on  $\partial \eta(\Om)$, \label{bry_p_full} \\
 \eta(0) = \id,~\dot{\eta}(0)=u_0.
\end{subnumcases}
\noindent where
 $\Om$ is a domain in $\RR^n$; $\eta(t,\cdot)$ is, for each $t$, a volume 
preserving embedding $\eta(t):\Om \rar \RR^n$ representing the fluid 
motion, with $t$ thought of as the time variable ($\eta(t, x)$ is the
 position at time $t$ of the fluid 
particle that at time
zero was at $x$); $``\,\dot{~}\,"$ denotes 
derivative with respect to $t$; 
$u: \Om(t) \rar \RR^n$ is a 
divergence free vector field on $\Om(t)$ defined by $u = \dot{\eta} \circ \eta^{-1}.$
It represents the fluid velocity. $\Om(t) = \eta(t)(\Om)$;
$\cA$ is the mean curvature of the boundary of the domain $\Om(t)$; $p$ is a real valued function
on $\Om(t)$
called the pressure; finally,
$k$ is a non-negative constant known as the coefficient of surface tension.
$\id$ denotes the identity map, $u_0$ is a given
divergence free vector field on $\Om$, 
 and 
$\operatorname{div}$ means divergence.
The unknowns are the fluid motion
$\eta$ and the pressure $p$, but notice that the system (\ref{free_boundary_full}) is coupled in a non-trivial fashion in the
sense that the other quantities appearing in (\ref{free_boundary_full}), namely $u$, $\cA$ and $\Om(t)$, depend 
explicitly or implicitly on $\eta$ and $p$.

We shall prove that in two (spatial) dimensions,
if the embeddings $\eta$ are sufficiently regular at the boundary
of $\Om$, if the boundary at time zero has constant mean curvature and if $u_0$ is tangent to the boundary, then, as $k \rar \infty$, solutions $(\eta, p$) to (\ref{free_boundary_full}) 
converge to solutions to the incompressible Euler equations on the fixed domain $\Om$, given by
\begin{subnumcases}{\label{Euler}}
 \ddot{\zeta}  = - \nabla p \circ \zeta,
\label{Euler_edo} \\
 \dive(\dot{\zeta}\circ \zeta^{-1}) = 0
\label{Euler_div}
\label{Euler_incompressible}  \\
\zeta(0) = \id,~\dot{\zeta}(0) = u_0
\label{Euler_initial_conditions}
\end{subnumcases}
\noindent where $\zeta(t, \cdot)$ is, for each $t$, a volume preserving 
diffeomorphism  $\zeta(t): \Om \rar \Om$.
It is well known that the pressure $p$ in the incompressible Euler equations is not an independent
quantity, since it is completely determined by the velocity vector field $\vartheta = \dot{\zeta}\circ\zeta^{-1}$ (see e.g. \cite{EM} or section \ref{final_proof}). 

In order to explain the aforementioned regularity of $\eta$ at 
the boundary (whose seeming necessity is discussed further below) as well as to state the main result, we need to introduce some definitions.

Given manifolds $M$ and $N$, denote by $H^s(M,N)$
the space of maps of Sobolev class $s$ between $M$ and $N$.
For $s > \frac{n}{2} +1 $ define
\begin{gather}
 \cE_\mu^s(\Om)= \cE_\mu^s  = \Big \{ \eta \in H^s(\Om,\RR^n) ~ \Big | ~ J(\eta) = 1, \eta^{-1} 
\text{ exists and belongs to }  H^s( \eta(\Om), \Om) \Big \},
\nonumber
\end{gather}
\noindent where $J$ is the Jacobian. $\cE_\mu^s(\Om)$ is 
therefore the space of $H^s$-volume-preserving embeddings of $\Om$ into $\RR^n$. 
Define also
\begin{gather}
 \cD_\mu^s(\Om)=  \cD_\mu^s = \Big \{ \eta \in H^s(\Om,\RR^n) ~ \Big | ~ J(\eta) = 1, \eta: \Om \rar \Om  
\text{ is bijective and $\eta^{-1}$ belongs to }  H^s \Big \},
\nonumber
\end{gather}
\noindent so that $\cD^s_\mu(\Om)$ is the space of $H^s$-volume-preserving diffeomorphisms of $\Om$.
Notice that $\cD_\mu^s(\Om) \subseteq \cE_\mu^s(\Om)$.

Let $\cB_{\de_0}^{s+\frac{1}{2}}(\partial \Om)$ be the open ball about zero  of radius $\de_0$ inside $H^{s+1/2}(\partial \Om)$. We shall prove that if $\de_0$ is sufficiently small, then the map
\begin{align}
& \varphi: \cB_{\de_0}^{s+\frac{1}{2}}(\partial \Om) \rar H^{s+1}(\Om), 
\nonumber \\
& \varphi(h) = f,
\nonumber
\end{align}
where $f$ satisfies 
\begin{subnumcases}{\label{jac_nl}}
J ( \id + \nabla f) = 1 & in $\Om$,  \label{jac_nl_int} \\
f = h & on $\partial \Om$, \label{jac_nl_bry}
\end{subnumcases}
is a well defined $C^1$ map, and $\varphi(\cB_{\de_0}^{s+\frac{1}{2}}(\partial \Om))$
is a smooth submanifold of $H^{s + 1 }(\Om)$.

 We note that the map $\varphi$ solves a non-linear analog of the Dirichlet problem; that of extending $h$ from $\partial \Om$ to a harmonic function on $\Om.$  In fact (\ref{jac_nl_int}) can be written $\Delta f + \det (D^2 f) = 0,$ so the difference between (\ref{jac_nl}) and the Dirichlet problem is only the term $\det(D^2 f).$  The purpose of (\ref{jac_nl_int})  to insure that $\id + \nabla f$ is volume preserving.   

  Using $\varphi$ we then 
construct another map
\begin{align}
\begin{split}
& \Phi: \cD_\mu^s(\Om) \times \varphi(\cB_{\de_0}^{s+\frac{1}{2}}(\partial \Om)) \rar \cE_\mu^s(\Om), 
 \\
{\rm defined \; by} \:\:\:\;& \Phi(\beta , f) = (\id + \nabla f) \circ \beta.
\end{split}
\label{big_phi}
\end{align}
Thus $\Phi(\beta,f)$ is the composition of two volume preserving maps.
We define $\mathscr{E}_\mu^s(\Om) \subseteq \cE_\mu^s(\Om)$ by
\begin{gather}
\mathscr{E}_\mu^s(\Om)
 = \Phi\big( \cD_\mu^s (\Om) \times \varphi(\cB_{\de_0}^{s+\frac{1}{2}}(\partial \Om))  \big).
\nonumber
\end{gather}
Notice that  since $\beta \in \cD_\mu^s(\Om)$, we have $\beta(\partial \Om) = \partial \Om$. 
Therefore, solutions $\eta$ to (\ref{free_boundary_full}) which belong
to $\ccE^s_\mu(\Om)$ decompose to a part fixing the boundary and
a boundary oscillation. This decomposition is one of the main ingredients 
of our proof. We shall also show that, under our hypotheses,  $\nabla f$ is in fact $1\half$ degree smoother
than $\eta$ (though $\beta$ is as regular as $\eta$). This is the sense in which we work with embeddings which
have smoother boundary values. The reason why one needs to consider
$\ccE_\mu^s(\Om)$ rather than $\cE_\mu^s(\Om)$ is explained in 
section \ref{regularity_spaces}.

We are now ready to state our main result.

\begin{theorem}
Let $\Om$ be a domain in $\RR^2$ with smooth boundary $\partial \Om$, and assume that
the mean curvature of $\partial \Om$ is constant. Given $k \in \RR_+$ and a divergence free
vector field $u_0 \in H^{5\half}(\Om, \RR^2)$ tangent to the boundary, let 
$(\eta_k, p_k) \in  C^1 \big ( [0,T_k), \cE_\mu^{4\half}(\Om) \big )
 \times C^0 \left([0,T_k), H^4(\Om(t))\right)$ be 
a solution to (\ref{free_boundary_full})
with initial condition $u_0$, and let
$\zeta \in C^1 \big ( [0,\infty), \cD_\mu^{4\half}(\Om) \big )  $ be a 
solution to 
(\ref{Euler}), also with initial 
condition $u_0$. Assume further 
that $\eta_k \in \ccE_\mu^{4 \half }(\Om)$.

Then there exist a $T > 0$ and a $k_0 \in \RR_+$, such that if $T_k$ is maximal, then 
$T_k \geq T$ for all $k \geq k_0$, and $\eta_k(t) \rar \zeta(t)$ as a $C^1$ 
curve in $\ccE^{4\half}_\mu(\Om)$ as $k \rar \infty$.
\label{main_theorem}
\end{theorem}
 
We stress that $\eta_k$, $p_k$ and $\zeta$ exist and are unique by \cite{CS, EM}, and
$\zeta$ is defined for all time by \cite{EbinSymp}.  The reason for requiring $u_0 \in H^{5\half}$, whereas solutions $\eta$ are only in $H^{4\half}$, is to 
ensure uniqueness (see \cite{CS}). We assume that the initial 
velocity $u_0$ is the same for all
 $k$ for simplicity. The result would still hold 
for appropriate sequences $\{ u_{0k} \}$ which converge, in a precise sense, 
to $u_0$ when $k \rar \infty$. Note that the assumption that $\partial \Om$ has constant curvature implies that it is connected, for if $\partial \Om$ had multiple components they would each have to be a circle and all would have the same radius.  Since $\Om$ is a bounded domain in the plane, one of these boundary components would also bound the non-compact component of $\RR^2 - \Om$ and this circle would have to have larger radius than the others.  Hence this circle must be the only component, so $\partial \Om$ must be connected.
However, the assumption of constant mean curvature at time zero is necessary.
 We comment
on these remarks, and other matters, in section \ref{final_remarks}.
We also point out that, in general, the convergence $\eta_k(t) \rar \zeta(t)$ is not expected to hold 
in $C^2$
even if the initial data are $C^{\infty}.$ To see this, pick any function $f$ which is constant on $\partial \Om$ and let $u_0 = (f_y, -f_x)$. Then $u_0$ will be divergence free and tangent to the boundary. The pressure for (1.2) at time zero will then satisfy:
$$ - \Delta p = 2(f_{xy}^2 - f_{xx}f_{yy})$$
and $\nabla_{\nu}p$ will equal zero on $\partial \Omega$.  Thus $p$ in general will not be constant on $\partial \Omega$, so one can not expect that
$\nabla p_k,$ the solution of (1.1), will converge to $\nabla p$, as $k \rar \infty$ even at time zero. 
 (see the analogous results in \cite{E2, ED}).

\begin{rema} Throughout the paper, we shall use the fact that the results in \cite{CS} also give $\dot{\eta} \in H^{4\half}(\Om)$.
\label{remark_regularity_eta_dot}
\end{rema}

\begin{notation} We reserve $\Om$ for the \emph{fixed} domain,
with $\Om(t)$ being always the domain at time $t$, i.e., $\Om(t) = \eta(t)(\Om)$. Of course,
$\Om(0) = \Om$. In several parts of the paper the subscript $k$ will be dropped for the sake of 
notational simplicity. 
\end{notation}

The mathematical study of equations (\ref{free_boundary_full}) has a long 
history, although for a long time
only results under restrictive conditions had been achieved. In particular,
a great deal of work has been devoted to irrotational flows, in which 
case the free boundary Euler equations reduce to the well-known
water-waves equations.
See \cite{Amborse, AmbroseMasmoudi, Craig, Lannes, Nalimov, Wu, Yosihara}.

Not surprisingly, when equations (\ref{free_boundary_full}) are considered in full generality, well-posedness becomes a yet more delicate issue, and most of the results are quite recent. In this regard, 
Ebin has showed that the problem is ill-posed if $k =0$ \cite{E0}, although
Lindblad proved well-posedness for $k = 0$ when the so-called ``Taylor sign condition'' 
holds \cite{Lin} (the linearized problem was also investigated by Lindblad in \cite{Lin2};
see also \cite{ChLin}).
When $k > 0$, a priori estimates have been obtained by
Shatah and Zeng \cite{ShatahZeng}, with well-posedness being finally 
established by Coutand and Shkoller \cite{CS, CSB} (see also \cite{Sch}). 
Other recent results, including the study of the compressible free boundary
Euler equations, are \cite{CS2, CS3, CSH, CSL}.

\begin{notation}
We use both $\nabla$ and $D$ to denote the derivative. $D_w$ is the directional 
derivative in the direction of $w$, $w$ a vector; in particular, 
with $\tau$ denoting the unit tangent vector field on $\partial \Om$, $D_\tau$
is the derivative along the boundary. We avoid using decimals to indicate fractional
derivatives as we find that it makes the text more difficult to read. Hence we employ $4\half$ to
denote $4.5$, etc.
\end{notation}

\subsection{The physical significance of theorem \ref{main_theorem} and the boundary regularity of (\ref{free_boundary_full}).\label{regularity_spaces}}

Theorem \ref{main_theorem} not only gives a satisfactory answer to 
the natural 
question of the dependence of solutions on the parameter $k$; it also 
addresses a well motivated problem in Applied Science, namely,
when one can, by considering a sufficiently high surface tension,
neglect the motion of the boundary in favor of the simpler description 
in terms of the equations within a fixed domain.

The physical intuition behind theorem \ref{main_theorem} is very simple, as we now explain.
The system (\ref{free_boundary}) can be derived from an action principle with Lagrangian
\begin{align}
 \cL(\eta) = K(\eta) - V(\eta) ,
\label{Lagrangian}
\end{align}
where
\begin{align}
 K(\eta, \dot{\eta}) = \frac{1}{2} \int_{\Om} |\dot{\eta}|^2
\label{kinetic}
\end{align}
is the kinetic energy and 
\begin{align}
 V(\eta) =  k |\partial \Om(t) | - k|\partial \Om(0)| = k \Big (  \operatorname{Length}(\partial \Om(t))
 -  \operatorname{Length}(\partial \Om(0)) \Big )
\label{potential}
\end{align}
is the potential energy\footnote{Many authors consider instead $V(\eta) = k |\partial \Om(t) |$.
As the equations of motion remain unchanged by adding a constant, 
we choose to normalize the potential energy to make $V=0$ at time zero.  Such a normalization
is convenient for our purposes as we are interested in taking $k \rar \infty$, in which 
case, if we did not subtract the contribution at time zero, $V(\eta)$ would diverge to infinity.}.

Our theorem 1.1 is almost an example of a general theorem on motion with a strong constraining force \cite{E2}.
For the general theorem we are given a Riemannian manifold $M$ and a submanifold $N$.  Also given is a function $V: M \rightarrow \RR$
which has $N$ as a strict local minimum in the sense that $\nabla V = 0$ on $N$ and $D^2 V$ is a positive definite bilinear form on the normal bundle of $N$ in $M$. Then if $\eta_k(t)$ is a motion given by the Lagrangian $\cL(\eta, \dot{\eta}) = \frac{1}{2} \langle \dot{\eta},\dot{\eta} \rangle -kV(\eta)$ where $\langle  \, , \, \rangle$ is the Riemannian metric,
 and if $\zeta(t)$ is a Lagrangian motion in $N$ of $\frac{1}{2} \langle \dot{\zeta},\dot{\zeta} \rangle$ with the same initial conditions as $\eta_k(t)$, the theorem says that $\eta_k(t)$ converges to $\zeta(t)$ as $k \rightarrow \infty$.  Also $\dot{\eta}_k \rightarrow \dot{\zeta}$, but the second derivative in general does not converge.  For our theorem, $M=\cE_\mu^{4\half}(\Om) ,$ $N=\cD_\mu^{4\half}(\Om) ,$ $\langle \;,\;\rangle$ is the $L^2$ inner product on tangent vectors and $V(\eta)$ is given by (1.7).

Our theorem 1.1 is not actually an example of the general theorem for two reasons:

a) The $L^2$ inner product on tangent spaces is only a weak Riemannian metric.  The topology that it induces is weaker than the  $H^{4\half}$ topology of
$\cE_\mu^{4\half}(\Om).$

b) The bi-linear form $D^2V$ is only weakly positive definite on each normal space; it gives the $H^{1\half}$ topology rather than the $H^{4\half}$ topology.

We now turn our attention to the ``boundary regularity" previously mentioned, 
i.e., to the seeming necessity of working with $\ccE_\mu^s(\Om)$ instead
of $\cE_\mu^s(\Om)$. As already pointed out,
theorem \ref{main_theorem} can be viewed as a generalization to infinite dimensions of
established results about the behavior of the
Euler-Lagrange equations near a submanifold which minimizes the potential energy. As in \cite{E2, E4, ED}, in the present work the manifold minimizing the potential
energy is $\cD^{4\half}_\mu(\Om)$. Thus it is 
natural to try a similar approach to that 
of \cite{E2, E4, ED}.
 Let us
outline what such an approach looks like, and then we shall point out 
why  it is problematic, and how the introduction of 
$\ccE_\mu^s(\Om)$ solves the problem. We stress that the arguments
sketched below are heuristic, their purpose being mainly a motivation
for the definition of $\ccE_\mu^s(\Om)$.

We would like to decompose the fluid motion as a part fixing the boundary and a boundary oscillation, and then produce estimates showing 
that the boundary oscillation goes to zero when $k \rar \infty$.
To this end, one first needs to show  that $\cE_\mu^s(\Om)$ is an infinite-dimensional Banach manifold  with $\cD_\mu^s(\Om)$ as an embedded 
submanifold. Following well known arguments \cite{E_manifold, E1, EM}, $\cE_\mu^s(\Om)$ can then be endowed with a Riemannian 
metric which is inherited by $\cD^s_\mu(\Om)$. As usual,
one expects the resulting metric to be only a week Riemannian metric\footnote{We recall 
 that a weak Riemannian metric is one which induces, on each tangent space, a weaker topology
 than the one given by the local charts. This is a feature exclusive to infinite dimensional 
 manifolds; see \cite{E1} for details.}.

We then 
seek  a decomposition of the form
\begin{align}
 T_\beta(\cE^s_\mu) = T_\beta \cD_\mu^s + \nu_\beta(\cD_\mu^s) = \dive^{-1}(0)_\nu \circ \beta + \nabla H^{s+1} \circ \beta
\label{decomp_tangents}
\end{align}
where $\beta \in \cD^s_\mu \subseteq \cE^s_\mu$, $\dive^{-1}(0)_\nu$ denotes $H^s$ 
divergence free vector fields
which are tangent to the boundary, and $\nu_\be(\cD_\mu^s)$ denotes the
normal space, at $\beta$, of $\cD_\mu^s(\Om)$ inside $\cE_\mu^s(\Om)$.
The next step would be to combine (\ref{decomp_tangents}) with an implicit function type of argument in order to obtain 
\begin{align}
 \cE_\mu^s \approx \nu(\cD_\mu^s) ~\text{ near } \cD^s_\mu,
\label{normal_bundle}
\end{align}
where $\nu(\cD_\mu^s)$ is the normal bundle of $\cD_\mu^s(\Om)$
 considered as a submanifold of
$\cE_\mu^s(\Om)$. From these considerations it should  follow that
elements $\eta \in \cE_\mu^s(\Om)$ which are sufficiently 
close to $\cD_\mu^s(\Om)$ can be written
as
\begin{align}
\eta = (\id + \nabla f) \circ \beta,
\label{def_f_beta_prior} 
\end{align}
giving the  decomposition.

Let us now see how the previous argument fails unless 
embeddings $\eta(\Om)$ such that $\partial \eta(\Om)$ is sufficiently regular are considered.
Suppose we succeed in constructing all the necessary manifold structures,
including the decomposition (\ref{decomp_tangents}). To
obtain the identification (\ref{normal_bundle}), one might use a  
 tubular neighborhood argument (see e.g. \cite{L}), which would rely on the exponential map.

 However in our infinite dimensional
setting, the existence of the exponential map is not immediately guaranteed, since
the metric is only a week Riemannian metric.
 In order to prove the existence
of geodesics, one has to derive a well-posedness result for the following second order ODE
on $\cE_\mu^s(\Om)$, which is obtained as critical point
of the kinetic energy for a curve $\theta: [0,T] \rar \cE_\mu^s(\Om)$,
\begin{gather}
\begin{cases}
\ddot{\theta} =
 \Big ( \nabla \Delta_0^{-1}
 \left( \dive \left( \nabla_{\dot{\theta}\circ \theta^{-1}}  
 \dot{\theta}\circ \theta^{-1} \right) \right) \Big ) \circ \theta & \text{ in } \Om, \\
 \theta(0) = \id,\, \dot{\theta}(0) = \nabla f,
 \end{cases}
\nonumber
\end{gather}
where $\Delta_0^{-1}$ means the inverse of the Laplacian (on the domain 
$\theta(\Om)$) with zero boundary
condition. To prove existence, one might attempt a Picard iteration argument, which requires the right hand side of the equation to be 
a (sufficiently) smooth function of its arguments, and this turns out not to be the case:
$\dot{\theta} \circ \theta^{-1} \in H^s$;
$\nabla_{\dot{\theta}\circ \theta^{-1}}  
 \dot{\theta}\circ \theta^{-1}  \in H^{s-1}$; then 
$\dive(\nabla_{\dot{\theta}\circ \theta^{-1}}  
 \dot{\theta}\circ \theta^{-1} )$ is also in $H^{s-1}$ since
 $\dot{\theta}\circ \theta^{-1}$ is divergence free; but because 
 $\theta \in H^s$, we have that $\theta \left|_{\partial \Om}\right. \in 
 H^{s-\frac{1}{2}}$, i.e., $\partial \theta(\Om)$ is only
 $H^{s-\frac{1}{2}}$ regular and therefore 
 $\Delta_0^{-1}(\dive(\nabla_{\dot{\theta}\circ \theta^{-1}}  \dot{\theta}\circ \theta^{-1} )) \in H^{s }$, 
from which follows 
 that $\nabla (\Delta_0^{-1}(\dive(\nabla_{\dot{\theta}\circ \theta^{-1}}  \dot{\theta}\circ \theta^{-1} )) )\in H^{s - 1}$, 
i.e., there is
 a loss of a derivative. Because of this, as is shown in \cite{E0}, the above system is not well-posed, so there is no exponential map. Notice that this has nothing to do with the 
regularity of $\partial \Om$ itself, as we have assumed that the domain $\Om$ has
 smooth boundary\footnote{Notice that this loss of a derivative is 
 an extra difficulty of the free boundary problem, not present in the 
 case of the Euler equations in a fixed domain. In this  situation, 
 solutions are geodesics in $\cD_\mu^s(\Om)$. Since the boundary is
 fixed, $\partial \Om$ is, say, smooth (or at least assumed to be as regular as the rest of the data of the problem), and then
 $\Delta_0^{-1}(\dive(\nabla_{\dot{\theta}\circ \theta^{-1}}  \dot{\theta}\circ \theta^{-1} )) \in H^{s + 1}$,
 so that  $\nabla (\Delta_0^{-1}(\dive(\nabla_{\dot{\theta}\circ \theta^{-1}}  \dot{\theta}\circ \theta^{-1} )) )\in H^{s}$. }\space \footnote{Of course, showing that $\ddot{\theta}$ is in the correct space is only
 a necessary condition to perform the Picard iteration, and does not by itself
 show that the ODE has a solution. The full proof has been 
carried out in the case of the Euler equations in a fixed domain
by Ebin and Marsden
 \cite{EM}, and more recently for geodesics on the group of symplectomorphisms by the second author \cite{EbinSymp}.}. 

Here we avoid these problems by working explicitly in $\ccE_\mu^s(\Om)$, 
which has the decomposition (\ref{def_f_beta_prior}) built into it.

\begin{notation}
When the manifolds are clear from the context, we 
write simply $H^s$ or $H^s(M)$ for $H^s(M,N)$.
Denote by $\n \cdot \n_s$ and $\n \cdot \n_{s,\partial}$
the Sobolev norms on $\Om$ and $\partial \Om$, respectively. In particular the
$L^2$ norms are  $\n \cdot \n_0$ and $\n \cdot \n_{0,\partial}$. 
By $H^s_0(\Om)$ we shall denote the space of $H^s$ functions on $\Om$ which vanish 
on $\partial \Om$.
\end{notation}

\subsection{Auxiliary results.\label{auxiliary}}
Here we recall some well known facts which will be used throughout 
the paper. For their proof, see e.g. \cite{BB, E1, P}.

\begin{prop} Let $s > \frac{n}{2} + 1$, $g \in \cD^s_\mu(\Om)$,
$f \in H^s(\Om)$. Then $f \circ g \in H^s(\Om)$ and 
\begin{gather}
 \p f \circ g \p_s \leq C \p f \p_s\left( 1 + \p g \p_s^s \right),
\nonumber
\end{gather}
where $C =C(n,s,\Om)$.
 \label{Sobolev_composition}
\end{prop}

We shall make use of the following well-known bilinear inequality
(see \cite{E1})
\begin{gather}
 \p u \, v \p_r \leq \, C \p u \p_r \p v \p_s,
\label{bilinear}
\end{gather}
for $s > \frac{n}{2}$, $s \geq r \geq 0$, where $C =C(n,s,r,\Om)$. 
Recall also that restriction to the boundary gives rise to 
a bounded linear map,
\begin{gather}
 \p u \p_{s, \partial} \leq C \p u \p_{s + \frac{1}{2}},~~ s > 0,
\label{restriction}
\end{gather}
with $C =C(n,s,\Om)$.

Next we recall the decomposition of a vector field into its gradient and divergence free part. Given 
a $H^s$ vector field $\omega$ on $\Om$, define the operator $Q: H^s(\Om, \RR^n) \rar \nabla H^{s+1}(\Om, \RR^n)$
by $Q(\omega) = \nabla g$, where $g$ solves
\begin{subnumcases}{}
 \Delta g = \dive(\omega) & in $\Om$, \nonumber \\
\frac{ \partial g}{\partial \nu} = \langle \omega, \nu \rangle & on $\partial \Om$.
\nonumber
\end{subnumcases}
Since solutions to the Neumann problem are unique up to additive constants, $\nabla g$ is uniquely determined by $\om$, so $Q$ is well defined.
Define $P: H^s(\Om,\RR^n) \rar \dive^{-1}(0)_\nu$, where $\dive^{-1}(0)_\nu$ denotes divergence free vector fields 
tangent to $\partial \Om$, by $P = I - Q$, where $I$ is the identity map. 
It follows that $Q$ and $P$ are orthogonal projections in $L^2$.

\begin{notation}
The letter $C$ will be used to denote several different constants. $C$ will never depend 
on $k$, $\eta$ or $p$. 
\end{notation}


\section{The space $\ccE_\mu^s(\Om)$\label{space_smoother_emb}.}
Here we construct the space $\ccE_\mu^s(\Om)$, as outlined in 
the introduction. We assume that $s > \frac{n}{2} + 1$, as usual.
We shall make repeated use of (\ref{bilinear}) to estimate the products involved.

To start we note that the equation
\begin{align}
 J(\id + \nabla f) = 1
 \nonumber
\end{align}
implies
\begin{gather}
 \Delta f + O \Big ((D^2f)^2 \Big ) = 0 .
\label{nl_Dirichlet_f}
\end{gather}
\noindent for small $\nabla f.$

Equation (\ref{nl_Dirichlet_f}) can be considered as a non-linear Dirichlet problem for $f$, and so for $f$ small, $f$ should be determined by its boundary values. 

We shall present the argument  for two dimensions, which is the
main case of interest in this work. The interested reader can generalized the construction 
of $\ccE_\mu^s(\Om)$ to higher dimensions.
In two dimensions equation (\ref{nl_Dirichlet_f})  becomes
$ \Delta f + f_{xx}f_{yy} -f_{xy}^2 = 0$, or equivalently
$ \Delta f +\det(D^2 f) = 0$, where $\det$ means the determinant.
Therefore, given $h \in H^{s+\frac{1}{2}}(\partial \Om)$, we 
are interested in solving 
\begin{subnumcases}{\label{nl_Dirichlet_f_2d_system}}
\Delta f + \det(D^2 f) = 0 & in $\Om$,
 \label{nl_Dirichlet_f_2d} \\
f = h & on $\partial \Om$. \label{nl_Dirichlet_bry}
\end{subnumcases}
Define a map 
\begin{align}
& F: H^{s+\frac{1}{2}}(\partial \Om) \times H^{s+1}(\Om) 
\rar H^{s+\frac{1}{2}}(\partial \Om) \times H^{s -1 }(\Om) ,
\nonumber \\
{\rm by} \:\:\:\:\:\: & F(h, f) = ( f\left|_{\partial \Om}\right. - h, \Delta f + \det(D^2 f) ).
\nonumber
\end{align}
Notice that $F$ is $C^1$ in the neighborhood of the origin and and $F(0,0) = 0$, where we denote by $0$ 
the origin in the product Hilbert space $H^{s+\frac{1}{2}}(\partial \Om) \times H^{s -1 }(\Om) $. Letting $w \in  H^{s+1}(\Om)$, we obtain
\begin{gather}
D_2 F(0,0)(w) = (  w\left|_{\partial \Om}\right., \Delta w ),
\end{gather}
where $D_2$ is the partial derivative of $F$ with respect to 
its second argument. From the uniqueness of solutions to the 
 Dirichlet problem 
it follows that $D_2 F(0,0)$ is an isomorphism, and therefore 
by the implicit function theorem there exists a neighborhood of zero
in $H^{s+\frac{1}{2}}(\partial \Om)$, which we can take without 
loss of generality as a ball
 $\cB_{\de_0}^{s+\frac{1}{2}}(\partial \Om)$, and 
a $C^1$
map $\varphi: \cB_{\de_0}^{s+\frac{1}{2}}(\partial \Om) \rar H^{s+1}(\Om)$
satisfying $\varphi(0) = 0$, and 
such 
that $F(h, \varphi(h)) = 0$ for all $h \in \cB_{\de_0}^{s+\frac{1}{2}}(\Om)$.
In other words, $f = \varphi(h)$ solves (\ref{nl_Dirichlet_f_2d_system}).

Furthermore, since $\varphi$ is continuous, given $\epsilon >0$ we
can choose $\de_0$ so small that 
$\p \varphi(h) \p_{s+1} < \epsilon$. But if $f$ is a solution of (\ref{nl_Dirichlet_f_2d_system}) with $\p f \p_{s+1} \leq \epsilon$, 
and 
$\epsilon$ is sufficiently small, then by elliptic theory the solution is unique 
and obeys the estimate
\begin{gather}
\p f \p_{s+1} \leq C \p h \p_{s+ \frac{1}{2}, \partial  },
\label{elliptic_estimate_f_bry} 
\end{gather}
where the constant $C$ depends only on $\epsilon$, $s$, and $\Om$.

Next, we show that the map $\varphi$ is injective. Let $f = \varphi(h)$
and $u = \varphi(g)$. Then $f-u$ satisfies 
\begin{gather}
\begin{cases}
\Delta (f-u) = - \det(D^2 f) + \det( D^2 u) & \text{ in } \Om,
 \\
f -u = h - g& \text{ on } \partial \Om,
\end{cases}
\label{f_u}
\end{gather}
But 
\begin{align}
\begin{split}
\det(D^2 u) - \det( D^2 f) & = 
u_{xx} u_{yy} - u_{xy}^2 -f_{xx} f_{yy} + f_{xy}^2 \\
& = (u_{xx} - f_{xx} ) u_{yy} + (f_{yy} - u_{yy}) f_{xx} 
+ (f_{xy} - u_{xy}) (f_{xy}+u_{xy} ),
\end{split}
\label{det_f_u}
\end{align}
and therefore combining (\ref{f_u}) with (\ref{det_f_u})
and standard elliptic estimates produces
\begin{gather}
\p f - u \p_{s+1} \leq 
C \Big ( \p f - u\p_{s+1} \big   (\p f \p_{s+1} + \p u \p_{s+1} \big ) +
\p h - g\p_{s + \frac{1}{2}, \partial} \Big ),
\label{estimate_f_u_intermediate}
\end{gather}
where $C$ depends only on $\Om$ and $s$.

By (\ref{elliptic_estimate_f_bry}) we can estimate 
$f$ and $u$ in terms of $h$ and $g$, respectively. Moreover,
since $\p h \p_{s+\frac{1}{2}, \partial} < \de_0$ and 
$\p g \p_{s+\frac{1}{2}, \partial }< \de_0$, we can choose $\de_0$ 
so small that
\begin{gather}
C  \p f - u\p_{s+1} \big   (\p f \p_{s+1} + \p u \p_{s+1} \big )
\leq \frac{1}{2} \p f - u\p_{s+1},
\nonumber
\end{gather}
and therefore this term can be absorbed on the left hand side 
of (\ref{estimate_f_u_intermediate}), yielding
\begin{gather}
\p f - u \p_{s+1} \leq 
C  \p h - g\p_{s + \frac{1}{2}, \partial} .
\nonumber 
\end{gather}
In particular, $f = u$ when $h = g$, so the map $\varphi$ is
injective.

Lastly, we show that the derivative of $\varphi$ (which we already know
exists and is continuous) is injective as well. From this and 
the above it then follows that $\varphi( \cB_{\de_0}(\partial \Om))$
is a submanifold of $H^{s+1}(\Om)$. We start computing 
the derivative of $\varphi$ at zero in 
the direction $z$. Since  $\varphi(0) = 0$,
\begin{gather}
\varphi(0 + z) - \varphi(0) = \varphi(z) = D\varphi(0)(z) + 
o(\p z\p_{s+\frac{1}{2}, \partial }).
\label{expansion_derivative_varphi}
\end{gather}
But $\varphi(z) = w$, where $w$ solves
\begin{gather}
\begin{cases}
\Delta w + \det(D^2 w) = 0 & \text{ in } \Om, \\
w = z & \text{ on } \partial \Om,
\end{cases}
\nonumber
\end{gather}
from which we can write
\begin{gather}
w = - \Delta_0^{-1} (  \det(D^2 w) ) + \cH(z),
\nonumber
\end{gather}
where $\Delta_0^{-1}$ is the inverse of the Laplacian with zero 
boundary condition and $\cH(z)$ is the harmonic extension of $z$ to 
the domain $\Om$. From standard properties of the Laplacian and 
(\ref{elliptic_estimate_f_bry}) we obtain
\begin{gather}
\p \Delta_0^{-1} (  \det(D^2 w) ) \p_{s+1} 
\leq C \p z \p_{s + \frac{1}{2},\partial }^2.
\nonumber
\end{gather}
From this estimate and the expansion 
(\ref{expansion_derivative_varphi}) we conclude
that 
\begin{gather}
D\varphi(0) = \cH.
\nonumber
\end{gather}
In particular, the derivative of $\varphi$ is injective at the origin.
Since $\varphi$ is $C^1$, we conclude that $D\varphi$ is one-to-one,
provided $\de_0$ is taken sufficiently small. 

Recall now the definition (\ref{big_phi}). Notice that
$\Phi$ is well defined (if $\de_0$ is small) and its image
belongs to $\cE_\mu^s(\Om)$ since $J(\beta) = 1$ and, by 
construction,  $J(\id + \nabla f) = 1$.

We have therefore proven:

\begin{prop}
Let $s > \frac{n}{2} + 1$ and let $B^{s+\frac{1}{2}}_{\de_0}(\partial \Om)$ be 
the open ball of radius $\de_0$ in $H^{s+\frac{1}{2}}(\partial \Om)$. Then
if $\de_0$ is sufficiently small, there exists an
embedding $\varphi: B^{s+\frac{1}{2}}_{\de_0}(\partial \Om) \rar H^{s+1}(\Om)$, 
given explicitly by $\varphi(h) = f$, where $f$ solves (\ref{nl_Dirichlet_f_2d_system}). Moreover, the map $\Phi$ given by (\ref{big_phi}) is well defined.
\label{embedding}
\end{prop}

\begin{defi}
Under the hypotheses of proposition \ref{embedding}, we define
$\mathscr{E}_\mu^s(\Om) \subseteq \cE_\mu^s(\Om)$ by
\begin{gather}
\mathscr{E}_\mu^s(\Om)
 = \Phi\big( \cD_\mu^s (\Om) \times \varphi(\cB_{\de_0}^{s+\frac{1}{2}}(\partial \Om))  \big).
\nonumber
\end{gather}
\end{defi}

\section{A new system of equations. \label{setting}}

In this section, we shall derive a different set of equations for 
the free boundary problem (\ref{free_boundary_full}). 
In light of proposition \ref{embedding} and the hypotheses of
theorem \ref{main_theorem} we can, from now on, assume that 
solutions $\eta$ to (\ref{free_boundary_full}) can be written as
\begin{align}
\eta = (\id + \nabla f) \circ \beta,
\label{def_f_beta} 
\end{align}
with $\beta \in \cD_\mu^s(\Om)$, $\nabla f \in H^s(\Om)$
and with $f$ satisfying (\ref{nl_Dirichlet_f_2d}).
We also observe that 
\begin{align}
 \beta(0) = \id,~~\nabla f(0) = 0.
 \nonumber
\end{align}

It is customary to write the pressure as a sum of an interior and a boundary term, namely,
 $p = p_0 + k \cA_H$,  so that the system (\ref{free_boundary_full}) takes the form
 \begin{subnumcases}{\label{free_boundary} }
  \ddot{\eta}  = - \nabla p \circ \eta = - (\nabla p_0 + k \nabla \cA_H)\circ \eta & in $\Om$,
\label{basic_fluid_motion} \\
  \Delta p_0 = - \dive (\nabla_u u) & in $\eta(\Om)$ , \\
\left. p_0 \right|_{\partial \eta(\Om)} = 0 & on $ \partial \eta(\Om)$,
\label{equation_p_0} \\
  \Delta \cA_H = 0 &  in  $\eta(\Om)$, \\
\left. \cA_H \right|_{\partial \eta(\Om)} =  \cA &   on $\partial \eta(\Om)$,
 \label{equation_harm_ext} \\
 \eta(0) = \id,~\dot{\eta}(0)=u_0.
\end{subnumcases}

The energy for the fluid motion (\ref{free_boundary}) is given by the sum
of the kinetic and potential energies (\ref{kinetic}) and (\ref{potential}), respectively,
\begin{align}
\begin{split}
 E(t) & = K(\eta,\dot{\eta}) + V(\eta) \\
& = \frac{1}{2}\int_\Om |\dot{\eta}|^2 + k \Big ( |\partial \Om(t)| - |\partial \Om| \Big )  ,
\label{energy}
\end{split}
\end{align}
with $\eta : [0,T) \rar \cE_\mu^s(\Om)$.
This energy is conserved, and therefore 
\begin{gather}
E(t)  = \frac{1}{2} \p u_0 \p^2_0 
\label{energy_conserved} 
\end{gather}
where we have used $\dot{\eta} = u \circ \eta$ and $\eta(0) = \id$. 

Differentiating (\ref{def_f_beta}) in time gives
\begin{align}
 \dot{\eta} = (\nabla \dot{f} + v \cdot D \,\nabla f + v) \circ \beta,
\label{dot_eta_f_beta}
\end{align}
where $v$ is defined by 
\begin{align}
 \dot{\beta} = v \circ \beta.
\label{definition_v}
\end{align}
Using $\eta(0) = \id$ and $\dot{\eta}(0) = u_0$, 
from (\ref{dot_eta_f_beta}) we obtain
\begin{align}
 u_0 = \nabla \dot{f}(0) + v_0.
\nonumber
\end{align}
Since $u_0$ is divergence free and tangent to the boundary, one sees that $\nabla \dot{f} (0) = 0$, and
so
\begin{gather}
 v_0 = P u_0 = u_0.
\nonumber
\end{gather}
Differentiating (\ref{dot_eta_f_beta}) again and using (\ref{basic_fluid_motion}) gives the following equation for 
$\nabla f:$
\begin{align}
 \nabla \ddot{f} + 2 D_v \nabla \dot{f} + D^2_{vv} \nabla f + (\dot{v} + \dot{v} \cdot \nabla v) D \, \nabla f
+ \dot{v} + v \cdot \nabla v = -\nabla p\circ (\id + \nabla f),
\label{equation_ddot_nabla_f}
\end{align}
where the operator $D^2_{vv}$, acting on a vector $w$, is given in coordinates by 
\begin{gather}
( D^2_{vv} w )^i = v^j v^l \partial_{j}\partial_{l} w^i,
\label{D_vv_two}
\end{gather}
or in invariant form by
\begin{gather}
D^2_{v v} w = D_{v} \nabla_{v} w - D_{\nabla_{v} v} w.
\nonumber
\end{gather}
Define $L$ on the space of maps from $\Om$ to $\RR^2$ by $L = \id +D^2f$, let $L_1 = PL$ and let $L_2 = QL$, where 
$P$ and $Q$ are as in section \ref{auxiliary}.  Notice that $L_1$ is invertible on the image of $P$ if $f$
is small, since in this case it will be close to the identity. 
Applying $Q\circ(\id - LL^{-1}P) = Q - L_2 L_1^{-1}P$ to (\ref{equation_ddot_nabla_f}) we obtain
\begin{subnumcases}{\label{system_f_v_full}}
\nabla \ddot{f}  + 2 (Q - L_2 L_1^{-1}P) D_v \nabla \dot{f}
 \nonumber \\
\; \; \; \; \; \;   +\, (Q - L_2 L_1^{-1}P)D^2_{vv} \nabla f 
= -(Q - L_2 L_1^{-1}P) (\nabla p\circ (\id + \nabla f) ) & in $\Om$,
\label{system_f_v_f_dot_dot} \\
 \dot{v} +  P(\nabla_v v) + 2 L_1^{-1}PD_v \nabla\dot{f} + L_1^{-1} P D^2_{vv} \nabla f = 0  & in $\Om$, \label{system_f_v_v_dot} \\
\nabla f (0) = \nabla \dot{f} (0) = 0,~~ v(0) = P u_0 = u_0. \label{initial conditions}
\end{subnumcases}
The system (\ref{system_f_v_full}) will be used to derive estimates for $\nabla f$ and $v$.

\section{Energy estimates.\label{energy_esimates_section}}

The main goal of this section is to derive energy estimates which will be used
to control the boundary oscillation $\nabla f$.

In what follows, we shall make use of the well known decomposition of  a vector
field into its gradient and divergence free parts, as presented at the end of the introductory section. Hence, recall that $Q: H^s(\Om, \RR^n) \rar \nabla H^{s+1}(\Om, \RR^n)$ and 
 $P: H^s(\Om,\RR^n) \rar \dive^{-1}(0)_\nu$
 (where $\dive^{-1}(0)_\nu$ denotes divergence free vector fields 
tangent to $\partial \Om$) are the operators realizing this decomposition. They
satisfy $P + Q = I$, where $I$ is the identity map, and 
since  $\nabla H^{s+1}(\Om, \RR^n)$ and $\dive^{-1}(0)_\nu$
are orthogonal, it follows that $Q$ and $P$ are orthogonal projections in $L^2$.

\begin{notation}
From the assumptions that  $\partial \Om \subset \RR^2$  has constant mean curvature,
 it follows that $\Om$ is a disk of fixed radius. In order to simplify the notation,
we shall assume once and for all that the radius is one. Therefore $\partial \Om = S^1$. Clearly there is no loss of generality in doing so.
\end{notation} 

Taking into account the previous remark on notation, denote
 by $\tau$ and $\nu$ the tangent and outer unit normal to $S^1=\partial \Om$, so that 
$\tau$ points counterclockwise-wise. One easily computes then
\begin{gather}
 D_\tau \tau = -\nu,
\nonumber
\end{gather}
and 
\begin{gather}
 D_\tau \nu = \tau
\nonumber.
\end{gather}

We shall also need the following. In two dimensions
\begin{gather}
 D\eta = 
\left(\begin{array}{cc}
 \partial_x \eta^1 & \partial_y \eta^1 \\
\partial_x \eta^2 & \partial_y \eta^2
\end{array}\right),
\nonumber
\end{gather}
so it follows that
\begin{gather}
( D\eta )^{-1}= 
\frac{1}{\partial_x \eta^1\partial_y \eta^2 - \partial_y \eta^1 \partial_x \eta^2}
\left(\begin{array}{cc}
 \partial_y \eta^2 & -\partial_y \eta^1 \\
- \partial_x \eta^2 & \partial_x \eta^1
\end{array}\right)
=
\left(\begin{array}{cc}
 \partial_y \eta^2 & -\partial_y \eta^1 \\
- \partial_x \eta^2 & \partial_x \eta^1
\end{array}\right),
\nonumber
\end{gather}
since the Jacobian determinant of $\eta$ is one. Therefore,
any Sobolev norm of $(D\eta)^{-1}$ is bounded by that of $D \eta$,
\begin{gather}
 \p (D\eta)^{-1} \p_s \leq C \p D\eta \p_s.
\label{D_eta_inverse}
\end{gather}
A similar statement
holds for $(D\beta)^{-1}$.
Moreover, changing variables, using $\beta(\Om) = \Om$ and $J(\beta)=1$ gives
\begin{gather}
 \int_\Om |\beta^{-1}|^2 = \int_{\beta\circ\beta(\Om)} |\beta^{-1}|^2 =
  \int_\Om |\beta|^2 .
\nonumber
\end{gather}
Hence
\begin{gather}
 \p \beta^{-1} \p_s \leq C \p \beta^{-1} \p_0 + C \p D\beta^{-1} \p_{s-1} 
\leq C \p \beta \p_0 + C \p D\beta^{-1} \p_{s-1}.
\nonumber
\end{gather}
Using again change of variables and $\beta(\Om) = \Om$ and $J(\beta)=1$,
\begin{gather}
 \p D\beta^{-1} \p_{s-1} \leq C \p (D\beta)^{-1} \p_{s-1}
\nonumber
\end{gather}
so that 
\begin{gather}
 \p \beta^{-1} \p_s \leq C \p \beta \p_s.
\label{beta_inverse_Sobolev}
\end{gather}

Given an initial condition $u_0 \in H^{5\half}(\Om)$ for the free boundary problem,
by the results of Coutand and Shkoller \cite{CS}, there exist a $T_k > 0$, a unique 
curve $\eta_k(t)  \in H^{4\half}(\Om)$, and a unique function $p_k(t)  \in H^4(\Om(t))$
satisfying the free boundary value problem (\ref{free_boundary}); 
moreover $\dot{\eta}_k \in H^{4\half}(\Om)$.

Consider the decomposition (\ref{def_f_beta}). As mentioned in the introduction, 
we shall prove that $\nabla f$ is in fact $1 \half$ degree smoother than
$\eta$; see proposition \ref{decomposition_more_regular}. Therefore
it  makes sense to consider $\nabla f_k \in H^s(\Om)$, $ 4\half \leq s \leq 6$. 
From  (\ref{initial conditions}) we know that 
$\nabla f_k$ and $\nabla \dot{f}_k$ are small 
for small time. More precisely we shall assume the following

\begin{assump} (Bootstrap assumption). There exists constant $C>0$, independent of $k$, 
 such that on some interval $[0,T_k)$,
\begin{subequations}
\begin{align}
 \p \nabla f_k \p_s \leq \frac{C}{k^\frac{6-s}{2}}, &~ 0 \leq s \leq 6,
\label{bootstrap_eq_1}
\\
 \p \nabla \dot{f}_k \p_s \leq \frac{C}{k^\frac{4.5-s}{2}},~ & 0 \leq s \leq 4\half.
\label{bootstrap_eq_2}
\end{align}
\label{bootstrap_eq}
\end{subequations}
\label{bootstrap_assump}
\end{assump}
Assumption \ref{bootstrap_assump} will be assumed in this section. Its validity
is shown in proposition \ref{decomposition_more_regular}, which is stated below and proven
in section \ref{regularity_and_elliptic_estimates}.

We seek to show that estimates (\ref{bootstrap_eq}) hold on a time interval
uniform in $k$; in other words, that $T_k$ does not become arbitrarily small as $k$ gets large. In order to 
do this we shall prove that (\ref{bootstrap_eq}) implies improved estimates for $\nabla f$ and 
$\nabla \dot{f}$, possibly with a different constant $C>0$.

\begin{rema}
For simplicity, in (\ref{bootstrap_eq_1}) and (\ref{bootstrap_eq_2}), all values $ 0 \leq s \leq 6$ and $0 \leq s \leq 4\half$ 
are considered. As we point out in section \ref{section_bootstrap},
it will be enough to have (\ref{bootstrap_eq}) for 
only finitely many values of $s$.
\label{remark_finite}
\end{rema}

\begin{rema}
Several estimates below are valid for $\p \nabla f_k \p_s$ and $\p \nabla \dot{f}_k \p_s$ sufficiently
small. In light of assumption \ref{bootstrap_assump}, and since we 
are interested in the limit $k\rar \infty$, there will be no loss of generality in assuming that such 
smallness conditions are in fact met. The most common application of this idea will be to drop 
higher powers of $\p \nabla f_k \p_s$ and $\p \nabla \dot{f}_k \p_s$, and also to keep only the highest
derivative terms. For example, if one has a long expression such as
$C\p \nabla f_k\p_3 (1 + \p \nabla f_k \p_4 ) + C\p \nabla \dot{f}_k \p_2^2 + C\p \nabla \dot{f}_k \p_{2\half}^2$, 
if $k$ is large enough this expression is controlled by 
$C(\p \nabla f_k \p_3 + \p \nabla \dot{f}_k \p_{2\half}^2)$.
When the norms dealt with are clear from the context, we shall
use a slight abuse of language and refer simply to $\nabla f_k$ or $\nabla \dot{f}_k$ as being small.
\end{rema}

The propositions below will be used in obtaining the desired energy estimates. Their
proofs are given in section \ref{regularity_and_elliptic_estimates}.

\begin{prop}
Let $\eta_k \in \ccE^{4\half}_\mu(\Om)$ be a solution of the 
free boundary problem (\ref{free_boundary})  as in 
theorem \ref{main_theorem}, and 
defined on a time interval $[0,T_k)$. Then $\dot{\eta_k} \in H^{4\half}(\Om)$.
Moreover, if $T_k$ is sufficiently small and $k$ sufficiently large, 
then  $\nabla f_k \in \cH^6(\Om)$, $\nabla \dot{f}_k 
\in H^{4\half}(\Om)$, and equations (\ref{bootstrap_eq}) hold.
\label{decomposition_more_regular}
\end{prop}
\begin{proof}
 See section \ref{regularity_and_elliptic_estimates}.
\end{proof}

It will also be important to relate equations in $\Om$ and $\Om(t)$. To this 
end, let $\xi$ denote either $\eta$ or $\id + \nabla f = \eta\circ \beta^{-1}$.
We denote by $\Delta_\xi$ the operator defined by
\begin{gather}
 \Delta_\xi g = \left( \Delta (g\circ \xi^{-1}) \right)\circ \xi,
\nonumber
\end{gather}
where $\Delta$ is the Laplacian in the domain $\xi(\Om)$.

\begin{prop} 
For sufficiently small time and small $\nabla f$, $\Delta_\xi$ is an invertible elliptic operator on $H^s_0(\Om)$ .
\label{perturbation_elliptic}
\end{prop}
\begin{proof}
 See section \ref{regularity_and_elliptic_estimates}.
\end{proof}

We are now ready to start our energy estimates. 
To facilitate the reading, most of the calculations will be split into several short lemmas, propositions and corollaries.

Define
\begin{gather}
 \widetilde{\eta} = \id + \nabla f.
\nonumber
\end{gather}
By construction, $\widetilde{\eta}$ is volume preserving, so the vector field
\begin{gather}
 \widetilde{u} = \dot{\widetilde{\eta}}\circ \widetilde{\eta}^{-1}
\label{u_tilde}
\end{gather}
is divergence free. Right invariance through $\beta$ implies that
$\widetilde{\eta}(\Om) = \eta(\Om)$, 
$\partial \widetilde{\eta}(\Om) = \partial \eta(\Om)$. In particular, $p$ is defined
on $\widetilde{\eta}(\Om)$.

Define the energy 
\begin{align}
 \begin{split}
  \widetilde{E}(t) & = \frac{1}{2} \int_\Om |\dot{\widetilde{\eta}}|^2 + k | \partial \widetilde{\eta}(\Om)| \\
& = \frac{1}{2} \int_\Om |\dot{\widetilde{\eta}}|^2 + k | \partial \eta(\Om)| .
 \end{split}
\label{energy_tilde}
\end{align}
The purpose of using $\widetilde{E}$ rather than the physical energy 
defined in (\ref{energy}) is that $\widetilde{E}$ allows us to estimate
$f$ more explicitly. Notice, however, that as  $\widetilde{E}$ does not
carry direct physical significance, we have not normalized it as to keep it
finite in the limit $k \rar \infty$.

\begin{lemma}
 \begin{gather}
 \widetilde{E}(t)   = \frac{1}{2} \p \widetilde{u}_0 \p^2_0 + k |\partial \Om| 
+ \int_0^t \int_\Om \langle \dot{\widetilde{\eta}},\, \ddot{\widetilde{\eta}} \rangle
+ \int_0^t \int_\Om \langle \dot{\widetilde{\eta}}, \, \nabla p \circ \widetilde{\eta} \rangle.
\nonumber
 \end{gather}
\label{lemma_AA}
\end{lemma}
\begin{proof} Using the formulas $|\partial \widetilde{\eta} (\Om)| = \int_{\partial \Om} | D_{\tau} \widetilde{\eta}|$ and 
$A\circ \widetilde{\eta} = -\frac{1}{|D_{\tau} \widetilde{\eta}|}\langle D_{\tau}
 (\frac{1}{|D_{\tau} \widetilde{\eta}|} D_{\tau} \widetilde{\eta}), N \rangle $ 
where $N$ is the normal to $\partial \widetilde{\eta}(\Om)$ and
$D_\tau$ is the derivative tangential to $\partial \Om$,
we directly compute
\begin{gather}
 \frac{d}{dt} | \partial  (\Om) | = 
\int_{\partial \Om} \cA \circ \widetilde{\eta} \, \langle \dot{\widetilde{\eta}}, N \circ \widetilde{\eta} \rangle 
\, |D_\tau \widetilde{\eta} |,
\nonumber
\end{gather}

With this we get
\begin{align}
\begin{split}
\int_{\widetilde{\eta}(\Om)} \langle \nabla p, \dot{\widetilde{\eta}} \circ \widetilde{\eta}^{-1} \rangle & = 
- \int_{\widetilde{\eta}(\Om)} p \dive( \dot{\widetilde{\eta}}\circ \widetilde{\eta}^{-1} ) +
\int_{\partial \widetilde{\eta}(\Om)} p \, \langle \dot{\widetilde{\eta}} \circ \widetilde{\eta}^{-1}, N \rangle \\
& = \int_{\partial \Om} p \circ \widetilde{\eta}  \, \langle \dot{\widetilde{\eta}}, N \circ \widetilde{\eta} \rangle 
\, |D_\tau \widetilde{\eta} |
= 
k \int_{\partial \Om} \cA \circ \widetilde{\eta}  \, \langle \dot{\widetilde{\eta}}, N \circ \widetilde{\eta} \rangle 
\, |D_\tau \widetilde{\eta} |  ,
\end{split}
\nonumber
\end{align}
where we have used $\dive( \dot{\widetilde{\eta}}\circ \widetilde{\eta}^{-1} ) = 0$ and
$p \circ \widetilde{\eta} = k \cA \circ \widetilde{\eta}$ on $\partial \Om$. Hence 
\begin{gather}
k \frac{d}{dt} | \partial \widetilde{\eta} (\Om) | = 
\int_{\widetilde{\eta}(\Om)} \langle \nabla p, \dot{\widetilde{\eta}} \circ \widetilde{\eta}^{-1} \rangle =
\int_{\Om} \langle \nabla p \circ \widetilde{\eta}, \dot{\widetilde{\eta}}  \rangle .
\nonumber
\end{gather}
Now the result is straightforward from the fundamental theorem of calculus and the definition of $\widetilde{E}$.
\end{proof}

\begin{lemma}
 \begin{gather}
\frac{1}{2k} \p \nabla \dot{f} \p_0^2 + |\partial \Om| + \frac{1}{2} \int_{\partial \Om} \left( |D_\tau Df|^2 - \langle \tau, \, D_\tau Df \rangle^2 \right) \leq 
\, \frac{1}{k}\widetilde{E}(t) - \int_{\partial \Om} \langle \tau, D_\tau Df \rangle 
+ \frac{1}{8} \int_{\partial \Om} | D_\tau Df|^4 \nonumber \\
 +   \frac{1}{2} \int_{\partial \Om} |D\tau Df|^2 \langle \tau, D_\tau Df\rangle
+ \int_{\partial \Om} R,
\nonumber
\end{gather}
where 
\begin{gather}
R = \frac{3}{16} \Big( |D_\tau Df|^2 + 2 \langle \tau,  D_\tau Df \rangle \Big)^3
\int_0^1 (1-t)^2 \Big( 1 + t \left( |D_\tau Df|^2 + 2 \langle \tau, D_\tau Df \rangle \right) \Big)^{-\frac{5}{2}} dt.
\nonumber
\end{gather}
\label{lemma_A}
\end{lemma}
\begin{proof}
 The energy can be written
\begin{gather}
 \widetilde{E}(t) =  \frac{1}{2}\int_\Om |\dot{\widetilde{\eta}}|^2 + k \int_{\partial\Om} |D_\tau \widetilde{\eta}|
= 
\frac{1}{2} \p \nabla \dot{f} \p_0^2 +  k \int_{\partial\Om} |D_\tau \widetilde{\eta}|.
\label{E_tilde_greater_D_tau_eta_tilde}
\end{gather}
But
\begin{align}
 \int_{\partial\Om} |D_\tau \widetilde{\eta}|  & = \int_{\partial\Om} | \tau + D_\tau Df | ,
\label{integral_D_tau_eta_tilde}
\end{align}
and direct computation gives
\begin{align}
\begin{split}
| \tau + D_\tau Df | & = \sqrt{  \langle \tau + D_\tau Df, \tau + D_\tau Df \rangle } \\
& = \sqrt{ 1 + |D_\tau Df|^2 + 2\langle \tau, D_\tau Df \rangle },
\label{direct_computation}
\end{split}
\end{align}
where we used $\langle \tau, \tau \rangle = 1$.
Using Taylor's theorem with integral remainder gives
\begin{align}
 \sqrt{ 1 + |D_\tau Df|^2 + 2\langle \tau, D_\tau Df \rangle }
  = & \, 1 + \frac{1}{2} \Big ( |D_\tau Df|^2 + 2 \langle \tau, D_\tau Df \rangle \Big ) \nonumber \\
& - \frac{1}{8} \Big ( |D_\tau Df|^2 + 2 \langle \tau, D_\tau Df \rangle \Big )^2 \nonumber 
 + R,
\nonumber
\end{align}
with $R$ as in the statement of the lemma.
From this, (\ref{E_tilde_greater_D_tau_eta_tilde}), (\ref{integral_D_tau_eta_tilde})
and (\ref{direct_computation}) the result follows.
\end{proof}

\begin{lemma}
 \begin{gather}
\int_{\partial \Om} (D_\tau D_\nu f)^2 \leq 
C \int_{\partial \Om} (D_\tau f)^2
+ 
C \int_{\partial \Om} \left( |D_\tau Df|^2 - \langle \tau, \, D_\tau Df \rangle^2 \right) 
\end{gather}
\label{lemma_B}
\end{lemma}
\begin{proof}
Writing
\begin{gather}
 Df = D_\tau f \tau + D_\nu f \nu
\label{Df_components}
\end{gather}
we obtain
\begin{gather}
 D_\tau Df = (D^2_\tau f + D_\nu f)\tau + (D_\tau D_\nu f - D_\tau f) \nu,
\label{D_tau_Df_components}
\end{gather}
from which it follows that
\begin{gather}
 |D_\tau Df|^2 = (D^2_\tau f + D_\nu f)^2 + (D_\tau D_\nu f - D_\tau f )^2,
\label{D_tau_f_square}
\end{gather}
and
\begin{gather}
 \langle \tau, D_\tau Df \rangle = D^2_\tau f + D_\nu f.
\label{D_tau_f_dot_tau}
\end{gather}
Using (\ref{D_tau_f_square}) and (\ref{D_tau_f_dot_tau}) one obtains
\begin{align}
\begin{split}
\int_{\partial \Om} \left( |D_\tau Df|^2 - \langle \tau, \, D_\tau Df \rangle^2 \right) =
\int_{\partial \Om}  (D_\tau D_\nu f - D_\tau f)^2 .
\nonumber 
\end{split}
\end{align}
The Cauchy inequality with $\epsilon$, i.e., $ab \leq \frac{\epsilon a^2}{2} + \frac{b^2}{2\epsilon}$,
$\epsilon > 0$,  gives
\begin{align}
 \begin{split}
\int_{\partial \Om}  (D_\tau D_\nu f - D_\tau f)^2
 & = \int_{\partial \Om}  \big ( (D_\tau D_\nu f)^2 + (D_\tau f)^2 - 2 D_\tau D_\nu f D_\tau f \big ) \\
& \geq \int_{\partial \Om} \big ( (D_\tau D_\nu f)^2 + (D_\tau f)^2 -\epsilon( D_\tau D_\nu f)^2 - \frac{1}{\epsilon} (D_\tau f)^2 \big ),
\nonumber
\end{split}
\end{align}
and the result follows by choosing $\epsilon$ sufficiently small.
\end{proof}

\begin{rema}
For the estimates below we shall fix a small $\ve > 0$. It will
be clear that these estimates in fact hold for any choice of $\ve > 0$ sufficiently small.
How small $\ve$ has to be is determined in section \ref{section_bootstrap}.
\label{rema_epsilon}
\end{rema}

\begin{lemma}
 Fix $\ve > 0$ as explained in remark \ref{rema_epsilon}. Then 
 \begin{gather}
\frac{1}{2k} \p \nabla \dot{f} \p_0^2 + |\partial \Om| + C \p D_\tau D_\nu f \p_{0,\partial}^2 \leq 
\frac{1}{k}\widetilde{E}(t)
\nonumber \\
+ C \left( \p D_\tau f \p_{0,\partial}^2 + \p Df \p_ 1^2
+ \p Df \p_{\frac{3}{2} + \ve,\partial} ( 1 + \p Df \p_{\frac{3}{2} + \ve,\partial} ) \p D_\tau D f \p_{0,\partial}^2 \right).
\nonumber
\end{gather}
\label{lemma_C}
\end{lemma}
\begin{proof}
Given any $\ve>0$, by the Sobolev embedding theorem, there exists a 
constant $C>0$, depending only on 
$\partial \Om$ and $\ve$, such that 
\begin{gather}
 \n D_\tau Df \n_{C^0(\partial \Om)} \leq C \n D_\tau Df \n_{\frac{1}{2} + \ve,\partial} 
\leq C \n Df \n_{\frac{3}{2} + \ve,\partial} .
\label{estimate_D_tau_Df_0_1}
\end{gather}
Use (\ref{estimate_D_tau_Df_0_1}) to find
\begin{gather}
\int_{\partial \Om} |D_\tau D f|^4 \leq \sup_{\partial \Om} |D_\tau D f|^2 \int_{\partial \Om} |D_\tau D f|^2
\leq C \n Df \n_{\frac{3}{2} + \ve,\partial}^2  \n D_\tau Df \n_{0,\partial}^2 .
\label{estimate_fourth_power}
\end{gather}
Analogously, 
\begin{align}
 \begin{split}
\int_{\partial \Om} |D_\tau D f|^2 \langle \tau , D_\tau Df \rangle 
 & \leq 
\int_{\partial \Om} |D_\tau D f|^3 \leq 
\sup_{\partial \Om} |D_\tau D f| \int_{\partial \Om} |D_\tau D f|^2 \\
& \leq C \n Df \n_{\frac{3}{2} + \ve,\partial} \n D_\tau Df \n_{0,\partial}^2 .
\label{estimate_third_power}
\end{split}
\end{align}
Next, we use integration by parts to estimate the term
\begin{gather}
 \int_{\partial \Om} \langle \tau,D_\tau Df \rangle = - 
\int_{\partial \Om} \langle D_\tau \tau, Df \rangle.
\nonumber
\end{gather}
Since the boundary is $S^1$, we see that 
$D_\tau \tau = -\nu$, $\nu$ being the outer unit normal, and hence the integral becomes
\begin{gather}
 \int_{\partial \Om} \langle \tau,D_\tau Df \rangle = 
\int_{\partial \Om} \langle \nu , Df \rangle 
= \int_{\partial \Om} D_\nu f .
\nonumber
\end{gather}
Now we will use the equation for $f$. Integrate (\ref{nl_Dirichlet_f_2d}) over $\Om$
and integrate the Laplacian term by parts to find
\begin{gather}
 \int_{\partial \Om} D_\nu f  = \int_{\Om} (f_{xy}^2 - f_{xx}f_{yy} )
\leq C \n D^2 f \n_0^2 \, \leq C \n Df \n_1^2.
\nonumber
\end{gather}
Therefore
\begin{gather}
 \Big | \int_{\partial \Om} \langle \tau,D_\tau Df \rangle  \Big | \leq 
C \n Df \n_{1 }^2 .
\label{estimate_normal}
\end{gather}
Now we estimate the remainder term from lemma \ref{lemma_A}.
By inequality (\ref{estimate_D_tau_Df_0_1}) and the bootstrap assumptions
we obtain that 
\begin{gather}
\sup_{ t \in [0,1]} \, \sup_{\partial \Om} \,
\frac{1}{ \Big( 1 + t \left( |D_\tau Df|^2 + 2 \langle \tau, D_\tau Df \rangle \right) \Big)^{\frac{5}{2}}} \leq C,
\label{integrand_term_R}
\end{gather}
provided that $k$ is sufficiently large. 
Using  $\langle \tau,  D_\tau Df  \rangle \leq |D_\tau Df |$ (since 
$|\tau|=1$), for any point on $\Om$ we find that
\begin{align}
\begin{split}
\Big( |D_\tau Df|^2 + 2 \langle \tau,  D_\tau Df  \rangle \Big)^3 
 & \leq C \sum_{\ell = 3}^6 |D_\tau Df|^\ell =
 C \sum_{\ell = 1}^4 |D_\tau Df|^{2 + \ell} 
  \\
& \leq C \sum_{\ell = 1}^4 |D_\tau Df|^2 \sup_{\partial \Om} |D_\tau Df|^\ell 
\end{split}
\label{non_integrand_term_R}
\end{align}
Hence from 
inequalities (\ref{integrand_term_R}) and (\ref{non_integrand_term_R}), by invoking
(\ref{estimate_D_tau_Df_0_1}) once more, we find that the remainder term $R$ in lemma \ref{lemma_A} obeys
 obeys the estimate
\begin{align}
 \begin{split}
\int_{\partial \Om} R
 \leq C \n Df \n_{\frac{3}{2} + \ve,\partial} \n D_\tau Df \n_{0,\partial}^2,
\label{estimate_remainder}
\end{split}
\end{align}
provided that $\nabla f$ is sufficiently small, i.e., that $k$ is large enough.
Using lemmas \ref{lemma_A} and \ref{lemma_B} along with 
(\ref{estimate_fourth_power}), (\ref{estimate_third_power}) and
(\ref{estimate_remainder}), produces the result.
\end{proof}

\begin{lemma}
 $L^{-1} P$ is a bounded linear operator in $H^s$, $s\geq 0$, if $\nabla f$ is sufficiently small.
\label{L_inv_P_bounded}
\end{lemma}
\begin{proof}
 Since $L = \id + D^2 f$, for $\nabla f$ small this operator is bounded, with bounded inverse on the image
of $P$. $P$ is an orthogonal projection in $L^2(\Om)$ and it is also bounded in $H^s(\Om)$, $s\geq 0$.
\end{proof}

\begin{lemma}
 \begin{gather}
 \frac{1}{2k} \p \nabla \dot{f} \p_0^2 + \p D_\tau D_\nu f \p_{0,\partial}^2  \leq 
 C \Big ( \p D_\tau f \p_{0,\partial}^2 + \p Df \p_1^2 
 \nonumber \\
 + \p Df \p_{\frac{3}{2} + \ve,\partial} ( 1 + \p Df \p_{\frac{3}{2} + \ve,\partial} ) \p D_\tau D f \p_{0,\partial}^2 \Big) 
\nonumber \\
 + \frac{C}{k}\int_0^t  \p \nabla \dot{f} \p_0 \p D^2_{vv} \nabla f \p_0  
 + \frac{C}{k}\int_0^t  \p \nabla \dot{f} \p_0 \p D f \p_{2+ \ve} \p  L_1^{-1} P \p 
\Big ( \p \nabla p \circ \widetilde{\eta} \p_0 \nonumber \\
 + \p D_v \nabla \dot{f} \p_0  + \p D^2_{vv} \nabla f   \p_0  \Big ), 
\nonumber
\end{gather}
where $\p  L_1^{-1} P \p $ is the $L^2$ operator norm of $L_1^{-1} P$,
$D^2_{vv}$ is given by (\ref{D_vv_two}), and $\ve>0$ is a small number
as indicated in remark \ref{rema_epsilon}.
\label{lemma_D}
\end{lemma}
\begin{proof}
First, we shall show that
 \begin{align}
\begin{split}
 \widetilde{E}(t)  &  =   k |\partial \Om| 
+ \int_0^t \int_\Om \langle \nabla \dot{f}, \, D^2 f L_1^{-1} P ( \nabla p \circ \widetilde{\eta})  \rangle 
- \int_0^t \int_\Om \langle \nabla \dot{f}, D^2_{vv} \nabla f \rangle  \\
& + 2 \int_0^t \int_\Om \langle \nabla \dot{f}, \, D^2 f L_1^{-1} P D_v \nabla \dot{f} \rangle 
+ \int_0^t \int_\Om \langle \nabla \dot{f}, \, D^2 f L_1^{-1} P D^2_{vv} \nabla f  \rangle .
\end{split}
\label{eq_CC}
\end{align}
Since $\dot{\widetilde{\eta}} = \nabla \dot{f}$, recalling that $L_2 = Q L$, $L = \id + D^2 f$ and 
using the fact that $Q$ is symmetric with respect to the $L^2$ inner product,
\begin{align}
\begin{split}
  \int_\Om \langle \dot{\widetilde{\eta}}, \, (Q - L_2 L_1^{-1} P )( \nabla p \circ \widetilde{\eta} ) \rangle 
=  \int_\Om \langle \nabla \dot{f}, \, (\id - (\id + D^2 f) L_1^{-1} P )(\nabla p \circ \widetilde{\eta}) \rangle  
 \\
=  \int_\Om \langle \nabla \dot{f}, \, \nabla p \circ \widetilde{\eta} \rangle  
- \int_\Om \langle \nabla \dot{f}, \, D^2 f L_1^{-1} P (\nabla p \circ \widetilde{\eta} ) \rangle , 
\end{split}
\label{eq_BB}
\end{align}
where we used the fact that $Q$ and $P$ are orthogonal and that $L_1^{-1}$ takes the image of $P$ onto itself.

Notice that $\widetilde{u}_0 = 0$ because of (\ref{initial conditions}) and (\ref{u_tilde}). Recall that
$\dot{\widetilde{\eta}} = \nabla \dot{f}$ and apply lemmas
\ref{lemma_AA} and (\ref{eq_BB}) to get
 \begin{align}
\begin{split}
 \widetilde{E}(t)  &  =   k |\partial \Om| 
+ \int_0^t \int_\Om \langle \nabla \dot{f}, \, D^2 f L_1^{-1} P ( \nabla p \circ \widetilde{\eta} ) \rangle 
+ \int_0^t \int_\Om \langle \nabla \dot{f}, \, \nabla \ddot{f} + (Q - L_2 L_1^{-1} P )(\nabla p \circ \widetilde{\eta} )\rangle .
\end{split}
\nonumber
\end{align}
We claim that
\begin{gather}
 \int_\Om \langle \nabla \dot{f},\, D_v \nabla \dot{f}\rangle = 0.
\label{Df_dot_Dv_Df_dot_zero}
\end{gather}
\label{lemma_D_v_f_inner_zero}
In fact, integrating by parts and using $\dive(v) = 0$ produces
\begin{align}
 \int_\Om \langle \nabla \dot{f},\, D_v \nabla \dot{f}\rangle  = 
- \int_\Om \langle D_v \nabla \dot{f},\, \nabla \dot{f}\rangle +
 \int_{\partial \Om} \langle \nabla \dot{f},\, \nabla \dot{f}\rangle  \langle \nu, v \rangle 
 = - \int_\Om \langle D_v \nabla \dot{f},\, \nabla \dot{f}\rangle,
\nonumber
\end{align}
since $\langle \nu, v \rangle = 0$. Using the form of $L_2$ and $L$, 
from (\ref{Df_dot_Dv_Df_dot_zero}), equation (\ref{system_f_v_f_dot_dot}), and
the properties of $Q$ and $P$ (as above), we obtain (\ref{eq_CC}).

From lemma \ref{lemma_C} and (\ref{eq_CC}), it follows that 
 \begin{align}
\begin{split}
\frac{1}{2k} \p \nabla \dot{f} \p_0^2 + \p D_\tau D_\nu f \p_{0,\partial}^2 & \leq 
 C \Big( \p D_\tau f \p_{0,\partial}^2 + \p Df \p_1^2 \\
& + \p Df \p_{\frac{3}{2} + \ve,\partial} ( 1 + \p Df \p_{\frac{3}{2} + \ve,\partial} ) \p D_\tau D f \p_{0,\partial}^2 \Big) \\
& + \frac{C}{k}\int_0^t \int_\Om \langle \nabla \dot{f}, \, D^2 f L_1^{-1} P \nabla p \circ \widetilde{\eta} \rangle 
- \frac{C}{k} \int_0^t \int_\Om \langle \nabla \dot{f}, D^2_{vv} \nabla f \rangle  \\
& + \frac{C}{k} \int_0^t \int_\Om \langle \nabla \dot{f}, \, D^2 f L_1^{-1} P D_v \nabla \dot{f} \rangle 
+ \frac{C}{k} \int_0^t \int_\Om \langle \nabla \dot{f}, \, D^2 f L_1^{-1} P D^2_{vv} \nabla f  \rangle .
\end{split}
\nonumber
 \end{align}
Now the lemma follows from the Cauchy-Schwarz inequality, (\ref{bilinear}) and the fact that $Q$, and hence $P$ and $L_1^{-1}$,
are bounded operators in $L^2$ (see lemma \ref{L_inv_P_bounded}).
\end{proof}

\begin{lemma}
 \begin{gather}
 \n D f \n_\frac{3}{2} \,\, \leq  C \left( \n D f \n_{2 + \ve} \n D f \n_\frac{3}{2} 
+ \n D_\nu f \n_{1, \partial} \right).
\nonumber
\end{gather}
Here, $\ve >0$ is as indicated in remark \ref{rema_epsilon}.
\label{lemma_E}
\end{lemma}
\begin{proof}
Since our goal is to estimate $\nabla f$, it can be assumed that $\int_{\partial \Om} f = 0$. 
Then one has the following elliptic estimate
\begin{gather}
 \n f \n_s \leq C \Big ( \n \Delta f \n_{s-2} + \n D_\nu f \n_{s - \frac{3}{2}, \partial} \Big ).
\label{elliptic_modulo}
\end{gather}
Use (\ref{nl_Dirichlet_f_2d}) and 
the Sobolev embedding theorem to obtain
\begin{align}
 \begin{split}
  \n \Delta f \n_\frac{1}{2} & \leq C \n (D^2 f)^2 \n_\frac{1}{2} 
\, \leq  C \n D^2 f \n_{C^0(\Om)}  \n D^2 f \n_\frac{1}{2}  
\, \leq C \n D^2 f \n_{1 + \ve} \n D^2 f \n_\frac{1}{2}  \\
& \leq C \n D f \n_{2 + \ve} \n D f \n_\frac{3}{2}  ,
 \end{split}
\nonumber
\end{align}
so that 
\begin{gather}
 \n f \n_\frac{5}{2} \leq C \left( \n D f \n_{2 + \ve} \n D f \n_\frac{3}{2} 
+ \n D_\nu f \n_{1, \partial} \right),
\nonumber
\end{gather}
from which the result follows.
\end{proof}

\begin{prop}
 Let $\cA$ be the  mean curvature of the boundary of the domain $\eta(\Om)$. Then
if $\nabla f$ is sufficiently small,
\begin{align}
\begin{split}
 \p \cA \circ \widetilde{\eta} - 1 \p_{s +\frac{1}{2},\partial} & \leq C \left(
\p D_\tau Df \p_{s^\prime+\frac{1}{2},\partial} + 
\p D_\tau^2 Df \p_{s^\prime+\frac{1}{2},\partial} \right). 
\end{split}
\nonumber
\end{align}
where $s^\prime = \max(s, \ve)$, and $\ve$ is a fixed small number 
as indicated in remark \ref{rema_epsilon}, and $0\leq s \leq 3$.
\label{prop_estimate_mean_curvature}
\end{prop}
\begin{proof}
Notice that the left hand side of the above inequality is well defined since $\eta(\Om)=\widetilde{\eta}(\Om)$.
The mean curvature at $\widetilde{\eta}(\Om)$ is given by
\begin{align}
 \begin{split}
 A\circ \widetilde{\eta} \, N \circ \widetilde{\eta} & = 
- \frac{1}{|D_\tau \widetilde{\eta}|}  D_\tau \left( \frac{1}{|D_\tau \widetilde{\eta}|} D_\tau \widetilde{\eta} \right) \\
  & = 
-\frac{1}{|D_\tau \widetilde{\eta}|^2}  D^2_\tau\widetilde{\eta} 
 +\frac{1}{|D_\tau \widetilde{\eta}|^4} \langle D_\tau^2 \widetilde{\eta},\, D_\tau \widetilde{\eta} \rangle D_\tau \widetilde{\eta},
 \end{split}
\nonumber
\end{align}
where $N$ is the normal to $\partial \widetilde{\eta}(\Om)$. Since $\widetilde{\eta} = \id + Df$, one has
\begin{gather}
 D_\tau \widetilde{\eta} = \tau + D_\tau Df,
\nonumber
\end{gather}
and 
\begin{gather}
 D_\tau^2 \widetilde{\eta} = -\nu + D_\tau^2 Df.
\nonumber
\end{gather}
Notice that by (\ref{estimate_D_tau_Df_0_1}), $|D_\tau \widetilde{\eta}|^{-1}$ is well 
defined for large $k$, and the same holds for the several expressions below derived from 
$|D_\tau \widetilde{\eta}|^{-1}$.
Then
\begin{gather}
 \langle D^2_\tau \widetilde{\eta},D_\tau \widetilde{\eta} \rangle = 
\langle -\nu,D_\tau Df \rangle + \langle \tau, D^2_\tau Df \rangle + \langle D^2_\tau Df, D_\tau Df \rangle,
\nonumber
\end{gather}
so that
\begin{align}
 \begin{split}
A\circ \widetilde{\eta} \, N \circ \widetilde{\eta} & = 
\frac{1}{|D_\tau \widetilde{\eta}|^2} \left( \nu - D_\tau^2 Df \right) \\
& + \frac{1}{|D_\tau \widetilde{\eta}|^4}  \Big ( 
 \langle - \nu,D_\tau Df \rangle + \langle \tau, D^2_\tau Df \rangle + \langle D^2_\tau Df, D_\tau Df \rangle \Big)
(\tau + D_\tau Df)
 \end{split}
\label{mean_curv_Df}
\end{align}
We compute
\begin{gather}
 |D_\tau \widetilde{\eta}|^2 = \langle \tau + D_\tau Df,   \tau + D_\tau Df\rangle
= 1 +  2 \langle  D_\tau Df, \tau \rangle + |D_\tau Df|^2 = 1 + M_0,
\nonumber
\end{gather}
where
\begin{gather}
  M_0 = 2 \langle  D_\tau Df, \tau \rangle + |D_\tau Df|^2 .
\label{M_0}
\end{gather}
But,
\begin{align}
\frac{1}{|D_\tau \widetilde{\eta}|^2} & = \frac{1}{1 + M_0} =
 1+ M_1,
\label{expansion_D_eta_tilde_2}
\end{align}
where
\begin{gather}
 M_1 = - M_0 + M_0^2 \int_0^1 \frac{1-t}{( 1 + t M_0 )^2} \, dt.
\label{M_1_int}
\end{gather}
Since 
$\frac{1}{|D_\tau \widetilde{\eta}|^4} = (1+M_1)^2$,
we have
\begin{align}
\begin{split}
 A\circ \widetilde{\eta} \, N \circ \widetilde{\eta} & = 
\nu + M_1 \nu - (1+M_1) D_\tau^2 Df \\
& + (1+M_1)^2 \Big ( \langle  -\nu, D_\tau Df \rangle + 
\langle \tau, D_\tau^2 Df \rangle + \langle D_\tau^2 Df, D_\tau Df \rangle \Big ) (\tau + D_\tau Df ) \\
& = \nu + M_1 \nu - (1+M_1) D_\tau^2 Df 
 + (1+M_1)^2 M_2 (\tau + D_\tau Df ), 
\end{split}
\nonumber
\end{align}
where 
\begin{gather}
M_2 =  \langle - \nu, D_\tau Df \rangle +
\langle \tau, D_\tau^2 Df \rangle + \langle D_\tau^2 Df, D_\tau Df \rangle .
\label{M_2} 
\end{gather}
From this it follows
\begin{gather}
 | \cA \circ \widetilde{\eta} |^2 = 1 + 2 \langle \nu, M_3 \rangle + M_3^2 =  1 + M_4,
\nonumber
\end{gather}
where 
\begin{gather}
 M_3  = M_1 \nu + (1+M_1) D_\tau^2 Df + (1+M_1)^2 M_2(\tau +  D_\tau Df ),
\label{M_3}
\end{gather}
and 
\begin{gather}
 M_4 = 2 \langle \nu, M_3 \rangle + M_3^2.
\label{M_4}
\end{gather}
We now have
\begin{gather}
 | \cA \circ \widetilde{\eta} | = \sqrt{ 1 + M_4 } 
 = 1 + M_5,
\nonumber
\end{gather}
where 
\begin{gather}
 M_5 = \frac{1}{2}M_4 -\frac{1}{4} M_4^2 \int_0^1 \frac{1-t}{(1+t M_4)^\frac{3}{2}} \, dt.
\label{M_5}
\end{gather}
Moreover, $ \cA \circ \widetilde{\eta}  > 0$ for large $k$, so we can drop the absolute value and write
\begin{gather}
 \cA \circ \widetilde{\eta} - 1 =  M_5.
\label{A_M_5}
\end{gather}
Our goal now is to use the above expression to estimate 
$\p \cA \circ \widetilde{\eta} - 1\p_{s+\frac{1}{2}, \partial}$. We shall use (\ref{bilinear})
to estimate the several products involved. If $s>0$, then the condition $s + \frac{1}{2} > \frac{n}{2} = \frac{1}{2}$ is satisfied;
if $s=0$ then we use $\p \cdot \p_{\frac{1}{2}, \partial} \leq C \p \cdot \p_{\frac{1}{2} + \ve, \partial}$
and apply (\ref{bilinear}) to the $\frac{1}{2} + \ve$ norm. Henceforth we write $s^\prime$, with $s^\prime$
defined as in the statement of the proposition, covering both situations. 

From (\ref{M_0}),
\begin{gather}
 \p M_0 \p_{s^\prime+\frac{1}{2},\partial} \leq 
C \left( 
\p D_\tau Df \p_{s^\prime+\frac{1}{2},\partial} + 
\p D_\tau Df \p_{s^\prime+\frac{1}{2},\partial}^2 \right) \leq \, C 
\p D_\tau Df \p_{s^\prime+\frac{1}{2},\partial},
\label{estimate_M_0_Df}
\end{gather}
where the last inequality holds if  $\nabla f$ is sufficiently small. Next, from
(\ref{M_1_int}),
\begin{align}
\begin{split}
\p M_1 \p_{s^\prime+\frac{1}{2},\partial} &  \leq C 
\p M_0 \p_{s^\prime+\frac{1}{2},\partial} + \p M_0 \p^2_{s^\prime+\frac{1}{2},\partial}
\Big \Vert \int_0^1 \frac{1-t}{( 1 + t M_0 )^2} \, dt  \Big \Vert_{s^\prime+\frac{1}{2},\partial} \\
& \leq C \p M_0 \p_{s^\prime+\frac{1}{2},\partial} \Big  ( 1 + 
\Big \Vert \int_0^1 \frac{1-t}{( 1 + t M_0 )^2} \, dt  \Big \Vert_{s^\prime+\frac{1}{2},\partial} \Big ),
\end{split}
\label{int_remainder_1}
\end{align}
where the last inequality holds if  $\nabla f$ is sufficiently small because of (\ref{estimate_M_0_Df}).
The Sobolev norm of the integrand can be estimated in terms of the Sobolev norm of
$M_0$ and the Sobolev norm of the real valued function $h_t(x) = \frac{1-t}{(1+tx)^2}$, $0 \leq t \leq 1$,
as follows. For $\de >0$
very small it holds that 
\begin{gather}
 \p h_t \p_{s,[-\de,\de]} \leq C,
\end{gather}
where $ \p h_t \p_{s,[-\de,\de]}$ is the Sobolev norm of $h_t$ over the interval $[-\de,\de]$ and the 
constant $C$ does not depend on $t$. But by (\ref{estimate_M_0_Df}), we have that
$h_t \circ M_0: \partial \Om \rar \RR$ satisfies $h_t \circ M_0 (x) \in (-\de,\de)$ for all $x \in \partial \Om$
(provided that $k$ is large enough),
hence (\ref{int_remainder_1}) gives
\begin{align}
\begin{split}
\p M_1 \p_{s^\prime+\frac{1}{2},\partial} & 
\leq C \p M_0 \p_{s^\prime+\frac{1}{2},\partial} \Big  ( 1 + 
C   \p M_0 \p_{s^\prime+\frac{1}{2},\partial} \Big ) \leq C  \p M_0 \p_{s^\prime+\frac{1}{2},\partial} \\
& \leq \, \p D_\tau Df \p_{s^\prime+\frac{1}{2},\partial},
\end{split}
\label{M_1}
\end{align}
where (\ref{estimate_M_0_Df}) has been employed. To estimate $M_2$, use (\ref{M_2}) to get
\begin{gather}
\p M_2 \p_{s^\prime+\frac{1}{2},\partial} \leq \, C \left(
\p D_\tau Df \p_{s^\prime+\frac{1}{2},\partial} + 
\p D_\tau^2 Df \p_{s^\prime+\frac{1}{2},\partial} \right) ,
\label{estimate_M_2_Df}
\end{gather}
where we ignored higher powers of the terms involved since we can assume that $k$ is 
sufficiently large and hence $\nabla f$ sufficiently small.

From (\ref{M_3}),
\begin{align}
\begin{split}
 \p M_3 \p_{s^\prime+\frac{1}{2},\partial} & \leq  \,
C \p M_1 \p_{s^\prime+\frac{1}{2},\partial}
+ C \Big (1+ \p M_1 \p_{s^\prime+\frac{1}{2},\partial} \Big ) \p D_\tau^2 Df \p_{s^\prime+\frac{1}{2},\partial} \\
& + C \Big (1+ \p M_1 \p_{s^\prime+\frac{1}{2},\partial} \Big )^2 \p M_2 \p_{s^\prime+\frac{1}{2},\partial}
 \Big (1+ \p D_\tau Df \p_{s^\prime+\frac{1}{2},\partial} \Big ) \\
& \leq C \left(
\p D_\tau Df \p_{s^\prime+\frac{1}{2},\partial} + 
\p D_\tau^2 Df \p_{s^\prime+\frac{1}{2},\partial} \right).
\end{split}
\label{estimate_M_3_Df_M_1_M_2}
\end{align}
where we have used (\ref{estimate_M_2_Df}) and (\ref{M_1}) and again dropped
higher powers of the terms involved.
From (\ref{M_4}) and (\ref{estimate_M_3_Df_M_1_M_2}) it then follows that
\begin{align}
\begin{split}
 \p M_4 \p_{s^\prime+\frac{1}{2},\partial} \leq C \left(
\p D_\tau Df \p_{s^\prime+\frac{1}{2},\partial} + 
\p D_\tau^2 Df \p_{s^\prime+\frac{1}{2},\partial} \right).
\end{split}
\label{estimate_M_4_M_3}
\end{align}
Finally $M_5$ is estimated from 
(\ref{M_5}) with the help of (\ref{estimate_M_4_M_3}); the integral term in (\ref{M_5}) 
is estimated by an argument similar to that used to obtain (\ref{M_1}); it yields
\begin{align}
\begin{split}
 \p M_5 \p_{s^\prime+\frac{1}{2},\partial} \leq C \left(
\p D_\tau Df \p_{s^\prime+\frac{1}{2},\partial} + 
\p D_\tau^2 Df \p_{s^\prime+\frac{1}{2},\partial} \right).
\end{split}
\label{estimate_M_5}
\end{align}
Combining (\ref{A_M_5}) and (\ref{estimate_M_5}) the result follows.
\end{proof}

In the estimates below, we shall employ some elliptic estimates derived from
proposition 
\ref{perturbation_elliptic}. These will involve a constant $C>0$
which a priori depends on $\eta$, $\widetilde{\eta}$ or $T_k$ (the
time interval where solutions are defined). The fact that these constants 
remain bounded in some uniform time interval is a consequence of the uniform estimates 
of sections \ref{section_bootstrap} and \ref{final_proof}.

\begin{prop}
 Let $\cA_H$ be the harmonic extension of the mean curvature to the domain $\eta(\Om)$. Then
if $\nabla f$ is sufficiently small,
\begin{gather}
 \p \nabla \cA_H \circ \widetilde{\eta} \p_s \leq C \p (D\widetilde{\eta})^{-1} \p_r
\left( \p D_\tau D f \p_{s^\prime + \frac{1}{2},\partial } + \p D^2_\tau Df \p_{s^\prime + \frac{1}{2}, \partial} \right),
\label{estimate_nabla_harmonic_extension}
\end{gather}
where $r = \max( s, 1+ \ve)$, $s^\prime = \max(s, \ve)$, and $\ve$ is a fixed small 
number  as indicated in remark \ref{rema_epsilon}, and $s =0, 1, 2, 3$.
\label{prop_harmonic_estimate}
\end{prop}
\begin{proof}
Notice that the left hand side of (\ref{estimate_nabla_harmonic_extension}) is well defined since 
$\eta(\Om) = \widetilde{\eta}(\Om)$. Consider the function $\widehat{\cA}_H: \Om \rar \RR$ defined
by $\widehat{\cA}_H = \cA_H \circ \widetilde{\eta} - 1$. It satisfies
\begin{gather}
 \begin{cases}
  \Delta_{\widetilde{\eta}} \widehat{\cA}_H = 0 & \text{ in } \Om, \\
\widehat{\cA}_H = \cA \circ \widetilde{\eta} - 1 & \text{ on } \partial \Om.
 \end{cases}
\nonumber
\end{gather}
Intuitively, $\widehat{\cA}_H$ is the harmonic extension of the mean curvature of the boundary of the moved domain minus the same function of the fixed domain, but parametrised by the fixed domain. Invoking proposition 
\ref{perturbation_elliptic} and standard elliptic estimates,
we have \begin{gather}
 \p \widehat{\cA}_H \p_{s+1} \leq C \p \cA \circ \widetilde{\eta} - 1 \p_{s+\frac{1}{2}, \partial}.
\label{elliptic_harm_ext_eta_tilde}
\end{gather}

Notice that proposition \ref{perturbation_elliptic} 
gives only the invertibility of $\Delta_{\widetilde{\eta}}$ in $H^s_0(\Om)$ (i.e.,
functions vanishing on the boundary). Estimate (\ref{elliptic_harm_ext_eta_tilde})
follows by combining $\Delta_{\widetilde{\eta}}^{-1}$ with the harmonic extension
of a function defined on $\partial \Omega$, what formally can be written 
as 
\begin{gather}
g = \Delta_{\widetilde{\eta}}^{-1} \Delta_{\widetilde{\eta}} g + 
\cH( \left. g \right|_{\partial \Om} ),
\label{ident_inv_bry}
\end{gather}
where $\cH$ denotes the harmonic extension.

Computing, 
\begin{gather}
\nabla \widehat{\cA}_H = \nabla \, ( \cA_H \circ \widetilde{\eta} - 1 ) = \nabla \cA_H \circ \widetilde{\eta} \, D\widetilde{\eta},
\nonumber
\end{gather}
hence,
\begin{gather}
 \nabla \cA_H \circ \widetilde{\eta} = \nabla \widehat{\cA}_H \, (D\widetilde{\eta})^{-1}.
\nonumber
\end{gather}
Using (\ref{bilinear}), (\ref{elliptic_harm_ext_eta_tilde}) implies
\begin{align}
\begin{split}
 \p \nabla \cA_H \circ \widetilde{\eta} \p_s & \leq C \p \nabla \widehat{\cA}_H \p_s \p (D\widetilde{\eta})^{-1} \p_r  \\
& \leq \, C \p \widehat{\cA}_H \p_{s+1} \p (D\widetilde{\eta})^{-1} \p_r  
\\
& \leq \, C \p \cA \circ \widetilde{\eta} - 1 \p_{s+\frac{1}{2}, \partial} \p (D\widetilde{\eta})^{-1} \p_r ,
\end{split}
\label{estimate_grad_A_comp_tilde}
\end{align}
where $r$ is as stated in the proposition. 
Combining (\ref{estimate_grad_A_comp_tilde}) with proposition \ref{prop_estimate_mean_curvature} gives the result.
\end{proof}

\begin{prop}
Let $v$ be defined as in (\ref{definition_v}), i.e., by $\dot{\beta} = v \circ \beta$. Then
\begin{gather}
 \p v \p_s \leq C,
\nonumber
\end{gather}
for small time and $s \leq 3$.
\label{bound_v}
\end{prop}
\begin{proof}
This is very much like the standard energy estimate for the Euler equations in 
the fixed domain $\Om$ \cite{MajBer, TayPDE}, except that the equation used
is (\ref{system_f_v_v_dot}), which contains the contributions
from $\nabla f$ and $\nabla \dot{f}$.
The estimate of $\p v \p_s$ then involves $\p \nabla \dot{f} \p_{s+1}$ and 
$\p \nabla f \p_{s+2}$.
These terms are bounded due the bootstrap assumption, provided that $s\leq 3$
(we could in fact take $s \leq 3 \half$, but $s \leq 3$ suffices for our purposes). 
\end{proof}

\begin{coro}
\begin{gather}
 \p \beta \p_s \leq C,
\nonumber
\end{gather}
for small time and $s \leq 3$.
\label{beta_bounded}
\end{coro}
\begin{proof}
 Since $\dot{\beta} = v \circ\beta$ and $\beta(0) = \id$, one has
\begin{gather}
 \beta = \id + \int_0^t v \circ \beta,
\nonumber
\end{gather}
and hence 
\begin{gather}
\p \beta \p_s \leq C + C \int_0^t \p v \p_s ( 1 + \p \beta \p_s ) \leq
C + C t + C \int_0^t \p \beta \p_s,
\nonumber
\end{gather}
where proposition \ref{bound_v}, proposition \ref{Sobolev_composition} and $J(\beta)=1$ have been employed. Iterating 
the above inequality yields the result.
\end{proof}

\begin{coro}
 \begin{align}
\p \eta \p_s \leq C, \nonumber \\
\p \dot{\eta} \p_s \leq C, \nonumber
 \end{align}
for small time and $s\leq 3$.
\label{eta_eta_dot_bounded}
\end{coro}
\begin{proof}
Use  (\ref{dot_eta_f_beta}),
proposition \ref{Sobolev_composition} and $J(\beta) = 1$ to find
\begin{align}
 \begin{split}
 \p \dot{\eta} \p_{s} & \leq \p \nabla \dot{f} \circ \beta \p_{s}
+ \p D_v \nabla f \circ \beta \p_{s} + \p v \circ \beta \p_{s}  \\
& \leq C\p \nabla \dot{f} \p_{s}\left( 1 + \p \beta \p_{s} \right) 
 + C \p v \p_{s} \p D^2 f\circ \beta \p_{s} + C \p v \p_{s}
\left( 1 + \p \beta \p_{s} \right) \\
&  \leq C\left( 1 + \p \beta \p_{s} \right) \left( \p v \p_{s} (1 + \p \nabla f \p_{s+1} ) 
+  \p \nabla \dot{f} \p_{s} \right) .
\end{split}
\nonumber
\end{align}
Proposition \ref{bound_v}, corollary \ref{beta_bounded} and the bootstrap assumption give
$ \p \dot{\eta} \p_s\leq C$. The bound on $\eta$ then follows from $\eta = \id + \int_0^t \dot{\eta}$.
\end{proof}

\begin{prop}
\begin{gather}
 \p \nabla p_0 \circ \eta \p_s \leq C,
\nonumber
\end{gather}
for small time and $s = 0, 1, 2, 3$.
\label{nabla_p_0_bounded}
\end{prop}
\begin{proof}
Define $q_0 = p_0 \circ \eta$ and notice that $\nabla p_0 \circ \eta = \nabla q_0 (D\eta)^{-1}$. Therefore
\begin{gather}
 \p \nabla p_0 \circ \eta \p_s \leq C \p (D\eta)^{-1} \p_{s^\prime} \p \nabla q_0 \p_s,
\label{nabla_p_0_eta}
\end{gather}
where $s^\prime = \max\{ s, 1 + \ve\}$. We claim that both terms on the right hand side can be estimated 
from $\dot{\eta}$.  In light of 
(\ref{D_eta_inverse}), the first term is estimated in terms of $\p D\eta \p_{s^\prime}$, and hence,
in terms of $\p \eta \p_{s^\prime+1}$. This last term, in turn, is estimated in terms of $\dot{\eta}$ due to 
\begin{gather}
\eta = \id + \int_0^t \dot{\eta},
\label{eta_int_eta_dot} 
\end{gather}
so
\begin{gather}
\p (D\eta)^{-1} \p_{s^\prime} \leq C (1 + \int_0^t \p \dot{\eta} \p_{s^\prime+1}).
\label{estimate_D_eta_inv_p_0}
\end{gather}
For the second factor on the right hand side of (\ref{nabla_p_0_eta}), notice that $q_0$ satisfies
\begin{gather}
 \begin{cases}
  \Delta_{\eta} q_0 = - (\partial_i u^j \partial_j u_i)\circ \eta  & \text{ in } \Om, \\
q_0= 0  & \text{ on } \partial \Om,
 \end{cases}
\nonumber
\end{gather}
so that, using 
standard elliptic estimates, the Sobolev embedding theorem, 
interpolation,  Cauchy's inequality, 
proposition \ref{perturbation_elliptic}, and 
a construction similar to (\ref{ident_inv_bry}), 
we have
\begin{align}
\begin{split}
 \p \nabla q_0 \p_s \leq \p q_0 \p_{s+1} & \leq C \p (\partial_i u^j \partial_j u_i)\circ \eta  \p_{s-1} \leq 
C \p \nabla u\circ \eta \p_{s^{\prime \prime}}^2 
\nonumber
\end{split}
\end{align}
where we let $s^{\prime\prime} = \max\{ s-1, 1+\ve\}$ and used (\ref{bilinear}).
Since $\nabla u \circ \eta = \nabla(u\circ\eta) (D\eta)^{-1}$, this implies
\begin{align}
\begin{split}
 \p \nabla q_0 \p_s & \leq C  \p \nabla(u\circ\eta) (D\eta)^{-1} \p_{s^{\prime \prime}}^2 
 \leq C \p \nabla(u\circ\eta) \p_{s^{\prime \prime}}^2 
\p (D\eta)^{-1}  \p_{s^{\prime \prime}}^2 \\
& \leq C \p u\circ\eta \p_{s^{\prime \prime}+1}^2 
(1 + \int_0^t \p \dot{\eta}  \p_{s^{\prime \prime}+1} )^2 \\
& \leq C \p \dot{\eta} \p_{s^{\prime \prime}+1}^2 
(1 + \int_0^t \p \dot{\eta}  \p_{s^{\prime \prime}+1} )^2
\end{split}
\label{estimate_nabla_q_0_eta_dot}
\end{align}
after invoking (\ref{bilinear}), (\ref{D_eta_inverse}) and (\ref{eta_int_eta_dot}), and 
recalling that $u\circ \eta = \dot{\eta}$. Using (\ref{estimate_D_eta_inv_p_0}) and (\ref{estimate_nabla_q_0_eta_dot})
with (\ref{nabla_p_0_eta}) yields
\begin{align}
\begin{split}
\p \nabla p_0 \circ \eta \p_s  
& \leq C \p \dot{\eta} \p_{s^{\prime \prime}+1}^2 
(1 + \int_0^t \p \dot{\eta} \p_{s^\prime+1})
(1 + \int_0^t \p \dot{\eta}  \p_{s^{\prime \prime}+1} )^2, 
\end{split}
\end{align}
so it is enough to estimate $\dot{\eta}$. 

Letting $r = \max\{ s^{\prime \prime}+1, s^{\prime}+1, 2 + \ve \}$, we see that 
$r \leq 3$ and hence the result follows from corollary \ref{eta_eta_dot_bounded}.
\end{proof}

\begin{prop}
 \begin{gather}
\p \nabla p \circ \widetilde{\eta} \p_0 \leq C\left( 1 + k \p \nabla f \p_{2 +\ve} + k \p \nabla f \p_{3 +\ve} \right) .
\nonumber
 \end{gather}
\label{nabla_p_bound}
 Here, $\ve>0$ is as indicated in remark \ref{rema_epsilon}.
\end{prop}
\begin{proof}
 From $p = p_0 + k \cA_H$ one obtains
\begin{gather}
 \p \nabla p \circ \widetilde{\eta} \p_0 \leq \p \nabla p_0 \circ \widetilde{\eta} \p_0 + 
k \p \nabla  \cA_H \circ \widetilde{\eta} \p_0  .
\nonumber
\end{gather}
Propositions \ref{Sobolev_composition} and \ref{nabla_p_0_bounded}, corollary \ref{beta_bounded} 
and $J(\beta) =1 $ imply
\begin{align}
\begin{split}
 \p \nabla p \circ \widetilde{\eta} \p_0 \leq 
\p \nabla p \circ \widetilde{\eta} \p_3 = \p \nabla p \circ \eta \circ \beta^{-1} \p_3
\leq \p \nabla p \circ \eta \p_3(1 + \p \beta^{-1} \p_3 ) \leq C.
\end{split}
\nonumber
\end{align}
Proposition \ref{prop_harmonic_estimate} gives
\begin{align}
\begin{split}
 \p \nabla \cA_H \circ \widetilde{\eta} \p_0 & \leq C \p (D\widetilde{\eta})^{-1} \p_{1+\ve}
\left( \p D_\tau D f \p_{\ve + \frac{1}{2},\partial } + \p D^2_\tau Df \p_{\ve + \frac{1}{2}, \partial} \right) \\
& \leq C \p \eta \p_{2+\ve}
\left( \p \nabla f \p_{2 + \ve } + \p \nabla f \p_{3 + \ve} \right).
\end{split}
\nonumber
\end{align}
Invoking (\ref{beta_inverse_Sobolev}) and corollary \ref{eta_eta_dot_bounded} we obtain the result.
\end{proof}

Finally we combine all of the above to obtain the desired energy estimate:

\begin{prop}
 \begin{align}
 \begin{split}
    \frac{1}{k} \p \nabla \dot{f} \p_0^2 + \p \nabla f \p_\frac{3}{2}^2 \, \leq\, &  C
   \p \nabla f \p_1^2 + 
\frac{C}{k} \int_0^t \p \nabla \dot{f} \p_0\p \nabla f\p_2  \\
&
  + \frac{C}{k} \int_0^t \p \nabla \dot{f} \p_0 \p \nabla f \p_{2+\ve} \left(\p \nabla \dot{f} \p_1 + \p \nabla f \p_2 \right) 
 \\
 &
 + \frac{C}{k} \int_0^t \p \nabla \dot{f} \p_0 \p \nabla f \p_{2+\ve} 
\left(1 + k\p \nabla f \p_{2+\ve} + k\p \nabla f \p_{3+\ve} \right)
\end{split}
 \label{energy_estimate_final}
 \end{align}
Here, $\ve>0$ is as indicated in remark \ref{rema_epsilon}.
\label{energy_estimate_final_prop}
\end{prop}
\begin{proof}
 From lemmas \ref{lemma_D} and \ref{L_inv_P_bounded}, propositions \ref{bound_v} and \ref{nabla_p_bound}, 
and (\ref{restriction}), it follows that 
 \begin{gather}
    \frac{1}{k} \p \nabla \dot{f} \p_0^2 + \p  D_\tau D_\nu f \p_{0,\partial}^2 \leq  
C \p \nabla f \p_{2+\ve} \p \nabla f\p_\frac{3}{2}^2 +
C \p \nabla f \p_1^2 + 
\frac{C}{k} \int_0^t \p \nabla \dot{f} \p_0\p \nabla f\p_2 \nonumber \\
  + \frac{C}{k} \int_0^t \p \nabla \dot{f} \p_0 \p \nabla f \p_{2+\ve} \left(\p \nabla \dot{f} \p_1 + \p \nabla f \p_2 \right) 
\nonumber \\
 + \frac{C}{k} \int_0^t \p \nabla \dot{f} \p_0 \p \nabla f \p_{2+\ve} 
\left(1 + k\p \nabla f \p_{2+\ve} + k\p \nabla f \p_{3+\ve} \right).
  \nonumber
 \end{gather}
Adding $\p  \nabla_\nu f \p_{0,\partial}^2 + \p \nabla f \p_{2+\ve}^2 \p \nabla f\p_\frac{3}{2}^2$ 
to both sides and 
invoking lemma \ref{lemma_E} yields
\begin{gather}
    \frac{1}{k} \p \nabla \dot{f} \p_0^2 + \p  \nabla f \p_\frac{3}{2}^2 \leq  
C \p  \nabla_\nu f \p_{0,\partial}^2 + C(1+ \p \nabla f \p_{2+\ve} ) \p \nabla f \p_{2+\ve} \p \nabla f\p_\frac{3}{2}^2 
\nonumber \\
+
C \p \nabla f \p_1^2 + 
\frac{C}{k} \int_0^t \p \nabla \dot{f} \p_0\p \nabla f\p_2 
  + \frac{C}{k} \int_0^t \p \nabla \dot{f} \p_0 \p \nabla f \p_{2+\ve} \left(\p \nabla \dot{f} \p_1 + \p \nabla f \p_2 \right) 
\nonumber \\
 + \frac{C}{k} \int_0^t \p \nabla \dot{f} \p_0 \p \nabla f \p_{2+\ve} 
\left(1 + k\p \nabla f \p_{2+\ve} + k\p \nabla f \p_{3+\ve} \right)
  \nonumber
 \end{gather}
The term $\p  \nabla_\nu f \p_{0,\partial}^2$ can be estimated by 
$\p  \nabla_\nu f \p_{\ve,\partial}^2$, which in turns is estimated by
$\p \nabla f \p_1^2$ because of (\ref{restriction}). Since 
$\p \nabla f \p_{2+\ve}$ is small due to the bootstrap assumption, the term
\begin{gather}
(1+ \p \nabla f \p_{2+\ve} ) \p \nabla f \p_{2+\ve} \p \nabla f\p_\frac{3}{2}^2 
\nonumber
\end{gather}
 can be absorbed 
into the term $\p \nabla f\p_\frac{3}{2}^2$ on the left hand side, finishing the proof.
\end{proof}

\section{Regularity of $\nabla f$ and elliptic estimates. \label{regularity_and_elliptic_estimates}}

In this section we shall prove propositions  and 
\ref{decomposition_more_regular} and \ref{perturbation_elliptic}. 

\begin{prop} Let $\eta$
and 
$p$ be as in theorem \ref{main_theorem}. Suppose further that 
assumption (\ref{bootstrap_eq_1}) holds for $s \leq 4\half$.
Then
 $\nabla f \in H^{6}(\Om)$.
\label{f_more_regular_v}
\end{prop}
\begin{proof}
The strategy will be to solve for the highest derivatives of $f$ in 
(\ref{mean_curv_Df}). Notice that $\beta$ is as regular as $\eta$.

Recalling (\ref{Df_components}) and (\ref{D_tau_Df_components}),
\begin{align}
 D^2_\tau Df = 
\left( D^3_\tau f + 2 D_\tau D_\nu f - D_\tau f \right) \tau
+ \left( D^2_\tau D_\nu f - 2 D_\tau^2 f - D_\nu f \right)\nu,
\nonumber
\end{align}
and so 
\begin{align}
\begin{split}
\langle D_\tau^2 Df, \tau + D_\tau Df \rangle & = 
\left( D_\tau^3 f +2 D_\tau D_\nu f - D_\tau f \right)
\left( 1 + D_\tau^2 f + D_\nu f \right) \\
& + \left( D_\tau^2 D_\nu f - 2 D_\tau^2 f - D_\nu f \right)
\left( D_\tau D_\nu f - D_\tau f \right).
\end{split}
\nonumber
\end{align}
Hence (\ref{mean_curv_Df}) gives
\begin{align}
\begin{split}
\langle \tau, \cA \circ \widetilde{\eta} N \circ \widetilde{\eta} \rangle & =  
-|D\widetilde{\eta}|^{-2}\left( D^3_\tau f + 2 D_\tau D_\nu f - D_\tau f \right)  \\
& - |D\widetilde{\eta}|^{-4}
 \Big(  \left( D_\tau^3 f +2 D_\tau D_\nu f - D_\tau f \right)
\left( 1 + D_\tau^2 f + D_\nu f \right) \\
& + \left( D_\tau^2 D_\nu f - 2 D_\tau^2 f - D_\nu f \right)
\left( D_\tau D_\nu f - D_\tau f \right) \\
& - \left( D_\tau D_\nu f - D_\tau f \right) \Big) \left( 1 + D_\tau^2 f + D_\nu f \right) 
.
\end{split}
\nonumber
\end{align}
We write this as
\begin{align}
\begin{split}
\langle \tau, \cA \circ \widetilde{\eta} N \circ \widetilde{\eta} \rangle & =  
-D_\tau^3 f \left( |D\widetilde{\eta}|^{-2} + |D\widetilde{\eta}|^{-4}\left( 1 + D_\tau^2 f + D_\nu f \right)^2 \right) 
 \\
& - D_\tau^2 D_\nu f   \left( 1 + D_\tau^2 f + D_\nu f \right) |D\widetilde{\eta}|^{-4} + \cR_1,
\end{split}
\label{dot_tau}
\end{align}
where $\cR_1$ is a polynomial
 expression
in at most two derivatives of $f$, and in $|D\widetilde{\eta}|^{-2}$. Analogously,
\begin{align}
\begin{split}
\langle \nu, \cA \circ \widetilde{\eta} N \circ \widetilde{\eta} \rangle & = 
-|D\widetilde{\eta}|^{-2} \left( D_\tau^2 D_\nu f - 2 D_\tau^2 f - D_\nu f \right) + |D\widetilde{\eta}|^{-2} \\
& - |D\widetilde{\eta}|^{-4}\Big ( 
\left( D_\tau^3 f +2 D_\tau D_\nu f - D_\tau f \right)
\left( 1 + D_\tau^2 f + D_\nu f \right) \\
& + \left( D_\tau^2 D_\nu f - 2 D_\tau^2 f - D_\nu f \right)
\left( D_\tau D_\nu f - D_\tau f \right) \\
& - \left( D_\tau D_\nu f - D_\tau f \right) \Big) \left( D_\tau D_\nu f - D_\tau f \right),
\end{split}
\nonumber
\end{align}
which can be written as 
\begin{align}
\begin{split}
\langle \nu, \cA \circ \widetilde{\eta} N \circ \widetilde{\eta} \rangle & = 
-D_\tau^2 D_\nu f \left( |D\widetilde{\eta}|^{-2} + |D\widetilde{\eta}|^{-4} \left( D_\tau D_\nu f - D_\tau f \right)^2 \right)\\
& - D_\tau^3  |D\widetilde{\eta}|^{-4} \left( 1 + D_\tau^2 f + D_\nu f \right) \left( D_\tau D_\nu f - D_\tau f \right) 
 |D\widetilde{\eta}|^{-4} + \cR_2,
\end{split}
\label{dot_nu}
\end{align}
where $\cR_2$ is a polynomial expression in at most two derivatives of $f$ and in $|D\widetilde{\eta}|^{-2}.$ 

To get a formula for $N\circ \widetilde{\eta}$ we note that the unit tangent vector to $\partial \widetilde{\eta}(\Om)$ is
$|D_{\tau} \widetilde{\eta} |^{-1} (\tau + D_{\tau} \nabla f).$  Thus if $\mathcal{J}$ denotes a clockwise right-angle rotation of the plane we find that $N\circ \widetilde{\eta} = |D_{\tau} \widetilde{\eta} |^{-1} \mathcal{J} (\tau + D_{\tau} \nabla f)$ which equals
$|D_{\tau} \widetilde{\eta} |^{-1} (\nu + \mathcal{ J} D_{\tau} \nabla f).$  Then 
letting $W= \mathcal{J} D_{\tau} \nabla f, $
we get 
\begin{gather}
 N \circ \widetilde{\eta} = |D_{\tau} \widetilde{\eta} |^{-1} (\nu + W).
 \label{normal}
\end{gather}
Use (\ref{normal}) to combine (\ref{dot_tau}) and 
(\ref{dot_nu}) as

\begin{subnumcases}{\label{algebraic_Df_3}}
 A D_\tau^3 f + r_1 D_\tau^2 D_\nu f = - \cA \circ \widetilde{\eta} \, \langle \tau,  W \rangle |D_\tau\widetilde{\eta}|^{-1} - \cR_1
\label{algebraic_Df_3_tau} \\
r_2  D_\tau^3 f + B D_\tau^2 D_\nu f = - \cA \circ \widetilde{\eta}|D_\tau\widetilde{\eta}|^{-1} 
 - \cA \circ \widetilde{\eta}|D_\tau\widetilde{\eta}|^{-1} \, \langle \nu,  W \rangle |D_\tau\widetilde{\eta}|^{-1}- \cR_2,
\label{algebraic_Df_3_nu}
\end{subnumcases}
where $r_1$ and $r_2$ are again
polynomials in at most two derivatives of $f$ and in $|D\widetilde{\eta}|^{-2}.$ 
 Also $A$ and $B$ are given by:
\begin{align}
A &= |D\widetilde{\eta}|^{-2} + |D\widetilde{\eta}|^{-4}\left( 1 + D_\tau^2 f + D_\nu f \right)^2,  \nonumber \\
B &= |D\widetilde{\eta}|^{-2} + |D\widetilde{\eta}|^{-4} \left( D_\tau D_\nu f - D_\tau f \right)^2. \nonumber
\end{align}
Notice that $W$ also contains at most two derivatives of $f$.

Given assumption \ref{bootstrap_assump} for $s \leq 4\half$, if $k$ is sufficiently large (and so $\nabla f$ sufficiently small),
use (\ref{expansion_D_eta_tilde_2}) to expand 
$|D\widetilde{\eta}|^{-2} = 1 + M_1$ and $|D\widetilde{\eta}|^{-4} = 1 + 2M_1 + M_1^2$,
as it was done in proposition 
\ref{prop_estimate_mean_curvature}.
It then follows that 
the lowest order terms in $r_1$ and $r_2$ are $O( \nabla f)$, and in particular $r_1$ and $r_2$ can
be made arbitrarily small by taking $k$ large. It also follows that $A$ and $B$ are nowhere zero,
so the algebraic system (\ref{algebraic_Df_3}) can be solved for $D_\tau^3 f $ and 
$D_\tau^2 D_\nu f$. We conclude that 

\begin{subnumcases}{\label{algebraic_sys_sol}}
D_\tau^3 f = \cP_1 (\cA \circ \widetilde{\eta}, \langle \tau, W\rangle ,
\langle \nu, W \rangle,\cR_3)
\label{polynomial_D_3_f_tau} \\
D_\tau^2 D_\nu f = \cP_2 (\cA \circ \widetilde{\eta}, \langle \tau, W\rangle ,
\langle \nu, W \rangle,\cR_4)
\label{polynomial_D_3_f_nu} 
\end{subnumcases}
where $\cP_1$ and $\cP_2$ are polynomials in their arguments and $\cR_3$ and $\cR_4$ are polynomial expressions
containing at most two derivatives of $f$ and in $|D_\tau \widetilde{\eta}|^{-1}$. The explicit dependence on $|D_\tau \widetilde{\eta}|^{-1}$ can be
eliminated by expanding $|D_\tau \widetilde{\eta}|^{-1} = 1 + M.$ Here $M$ is a smooth function of $D^2 f$ which is $O(D^2 f).$
Hence this term is absorbed into $\cR_3$ 
and $\cR_4.$

We can now determine the regularity of $f$. First notice that $\beta$ is as regular as $\eta$. 
By construction, we know that $\nabla f$ is at least in $H^{4\half}(\Om).$
From the hypotheses we have $\nabla p \circ \eta \in H^3(\Om)$
and $\nabla p_0 \circ \eta \in H^3(\Om)$, hence $\nabla \cA_H \circ \eta \in H^3(\Om)$,
and $\nabla \cA_H \circ \widetilde{\eta} = \nabla \cA_H \circ \eta \circ \beta^{-1} \in H^3(\Om)$,
where we recall that $p_0$ and $\cA_H$ have been defined in (\ref{free_boundary}). Since
\begin{gather}
 \nabla (\cA_H \circ \widetilde{\eta} ) = \nabla \cA_H \circ \widetilde{\eta}\,  D\widetilde{\eta}
\nonumber
\end{gather}
one sees that $\nabla (\cA_H \circ \widetilde{\eta} ) \in H^3(\Om)$, so 
$\cA_H \circ \widetilde{\eta} \in H^4(\Om)$,  and therefore
$\cA \circ \widetilde{\eta} \in H^{3\half}(\partial \Om)$.
We infer from (\ref{algebraic_sys_sol}) --- using also the fact that the right hand side of this expression
has at most two derivatives of $f$ --- that $D_\tau f \in H^{5\half}(\partial \Om)$
and $D_\nu f \in H^{5\half}(\partial \Om)$, which implies
$\nabla f \in H^6(\Om)$. \hfill $\qed$\\ 

\noindent \emph{Proof of proposition \ref{decomposition_more_regular}:}
We know already that $\dot{\eta}_k \in H^{4\half}(\Om)$ (see remark \ref{remark_regularity_eta_dot}), so that  $\nabla f_k$ 
and $\nabla \dot{f}_k$ are in $H^{4\half}(\Om)$. As they satisfy
$\nabla f_k(0) = 0 = \nabla \dot{f}_k(0)$, shrinking $T_k$ if necessary we can assume 
that assumption \ref{bootstrap_assump} holds for $s \leq 4\half$.
We can therefore apply proposition \ref{f_more_regular_v} and conclude that
in fact
$\nabla f_k \in H^6(\Om)$. Shrinking $T_k$ once more, we can now assume
that assumption \ref{bootstrap_assump} holds in full, finishing the proof.
\end{proof}

\noindent \emph{Proof of proposition \ref{perturbation_elliptic}:} Either $\xi = \id + \nabla f$ or
$\xi = \beta + \nabla f \circ \beta = \id + \int_0^t v\circ \beta + \nabla f \circ \beta$. $\p\nabla f\p_3$ will be 
small by the bootstrap assumption; for small time $\int_0^t \p v\circ \beta \p_3$ will be small due to 
proposition \ref{bound_v} and corollary \ref{beta_bounded}, whereas $\p \nabla f \circ \beta \p_3$ will be small in light of 
proposition \ref{Sobolev_composition}, the bootstrap assumption and corollary \ref{beta_bounded}.
Therefore $\xi$ is near the identity in $H^3$ and so $\Delta_{\xi}$ is a bounded elliptic operator
from $H^s_0(\Om)$ to $H^{s-2}(\Om).$ for $s \leq 3.$
\hfill $\qed$\\ 

\begin{rema} Notice that $\Delta_{\xi}$ in the previous
proof is bounded uniformly in $\xi$ as long as the bootstrap assumption holds. We shall show in the next section that equations (\ref{bootstrap_eq}) in fact hold in 
a uniform time interval, so the perturbation will be uniformly bounded in a small time interval independent of $k$.
\end{rema}

\section{The bootstrap estimate.\label{section_bootstrap}}
The goal of this section is to show 
that assumption \ref{bootstrap_assump} holds uniformly on a small time interval. 
As pointed out in remark \ref{remark_finite}, it is enough to our purposes to prove this
for finitely many values of $s$. 

\begin{prop}
Assume that (\ref{bootstrap_eq_1}) holds for $s=0,1,2,2+\ve,3+\ve \;{\rm and}\; 6$, and that (\ref{bootstrap_eq_2})
holds for $s=0,1 \;{\rm and}\; 4\half$. Then there exist a number $\La > 0$ and a constant $\widetilde{C}>0$,
independent of $k$, such 
that 
\begin{subequations}
\begin{align}
 & \p \nabla f_k \p_s \leq \frac{\widetilde{C}}{k^{\frac{6-s}{2} + \La} }, ~ s=0,1,2,2+\ve,3+\ve,6,
\label{bootstrap_improved_eq_1} \\
 & \p \nabla \dot{f}_k \p_s \leq \frac{\widetilde{C}}{k^{\frac{4.5-s}{2} + \La}},~ s=0,1, 4\half.
\label{bootstrap_improved_eq_2}
\end{align}
\label{bootstrap_improved_eq}
\end{subequations}
\label{prop_energy_bootstrap}
\end{prop}

\noindent \emph{Proof for $\p \nabla f \p_s$, $s=0,1,2,2+\ve,3+\ve$:}
We will use an iteration scheme, hence, write assumption
(\ref{bootstrap_eq_1}) as
\begin{align}
 \p \nabla f \p_s \leq \frac{C}{k^{\frac{6-s}{2} + \frac{\de_j(s)}{2}}}, 
\label{bootstrap_eq_1_general}
\end{align}
for $s=0,1,2,2+\ve,3+\ve,6$, where $\frac{\de_j(s)}{2}$ is defined inductively by $\frac{\de_0(s)}{2} = 0$ and
\begin{subequations}
\begin{align}
&\frac{\de_{j+1}(1.5)}{2} = \min\Big \{ \frac{0.5 + \de_{j}(1)}{2}, \frac{ 1.5 + \de_{j}(2+\ve) - \ve}{4} \Big \}, 
\label{delta_def_1.5} \\
 &\frac{\de_{j+1}(1)}{2} = \frac{\de_{j+1}(1.5)}{3} 
\label{delta_def_1} \\
&\frac{\de_{j+1}(2)}{2} = \frac{4\de_{j+1}(1.5)}{9}, 
\label{delta_def_2} \\
&\frac{\de_{j+1}(2+ \ve)}{2} = \frac{(4-\ve) \de_{j+1}(1.5)}{9}, 
\label{delta_def_2_ep} \\
&\frac{\de_{j+1}(3+ \ve)}{2} = \frac{(3 - \ve)\de_{j+1}(1.5)}{9}.
\label{delta_def_3_ep} 
\end{align}
\label{delta_def}
\end{subequations}

Notice that with the exception of the first line, $j+1$, and not $j$, appears on the right hand side of the 
above expressions.

Plugging (\ref{bootstrap_eq_1_general}) into (\ref{energy_estimate_final}) yields
\begin{align}
\begin{split}
  \frac{1}{k}\p \nabla \dot{f} \p_0 ^2+ \p \nabla f \p_\frac{3}{2}^2 & \leq
\frac{C}{k^{5 + \de_j(1)}} + \frac{C}{k^{\frac{10.5}{2} + \frac{\de_j(2)}{2}}} 
\\
& + \frac{C}{ k^{ \frac{10.5 - \ve}{2} + \frac{ \de_j(2+\ve) }{2} } } 
\left( \frac{1}{k^{ \frac{ 5.5 }{2} }} + \frac{1}{ k^{2+ \frac{\de_j(2)}{2}}} +1 + 
\frac{1}{k^{\frac{1}{2} -\frac{\ve}{2} + \frac{\de_j(3+\ve)}{2} } } \right).
\end{split}
\nonumber
\end{align}
For $\ve$ sufficiently small (see remark \ref{rema_epsilon}), the term in parenthesis on the right hand side is bounded by a constant $C$.
It then follows that the second term on the right hand side is bounded by the third term because
from the above definitions one has 
$\frac{10.5 - \ve}{2} + \frac{ \de_j(2+\ve) }{2}  \leq \frac{10.5}{2} + \frac{\de_j(2)}{2}$. Therefore
\begin{align}
\begin{split}
  \frac{1}{k}\p \nabla \dot{f} \p_0^2 + \p \nabla f \p_\frac{3}{2}^2  & \leq
\frac{C}{k^{5 + \de_j(1)}}  + \frac{C}{ k^{ \frac{10.5 - \ve}{2} + \frac{ \de_j(2+\ve) }{2} } } \\
& = \frac{C}{k^{4.5 + 0.5 + \de_j(1)}}  
+ \frac{C}{ k^{ 4.5 +  \frac{ 1.5 + \de_j(2+\ve) -\ve}{2} }  }.
\end{split}
\label{energy_with_k}
\end{align}
In particular,
\begin{gather}
 \p \nabla f \p_\frac{3}{2}  
\leq \frac{C}{k^{\frac{4.5}{2} + \frac{0.5 + \de_j(1)}{2} } }  
+ \frac{C}{ k^{ \frac{4.5}{2} + \frac{ 1.5 + \de_j(2+\ve) -\ve}{4} }  }
\leq \frac{C}{k^{\frac{4.5}{2} + \frac{\de_{j+1}(1.5)}{2} } }.
\label{improved_1.5}
\end{gather}
In order to handle the other values of $s$, recall the interpolation inequality
\begin{gather}
 \p u \p_s^{t-r}  \leq C\p u \p_t^{s-r} \p u \p_r^{t-s}\; {\rm for}~r \leq s \leq t,
\label{interpolation}
\end{gather}
which implies
\begin{gather}
 \p \nabla f \p_s^{t-1.5}  \leq C\p \nabla f \p_t^{s-1.5} \p \nabla f \p_{1.5}^{t-s}\; {\rm for} ~ 1.5 \leq s \leq t \leq 6,
\label{interp_greater}
\end{gather}
and 
\begin{gather}
 \p \nabla f \p_s^{1.5-r}  \leq C\p \nabla f \p_{1.5}^{s-r} \p \nabla f \p_r^{1.5-s} \; {\rm for} ~ 0 \leq r < s < 1.5.
\label{interp_less}
\end{gather}
Using (\ref{bootstrap_eq_1_general}), (\ref{improved_1.5}), and (\ref{interp_greater}) 
 with $t=6$ yields
\begin{align}
 \p \nabla f \p_s \leq \frac{C}{k^{\frac{6-s}{2} + \frac{\de_j(s)}{2}}}, 
\label{improved_greater}
\end{align}
for $s = 2, 2+ \ve, 3 + \ve$, with $\de(2)$, $\de(2+\ve)$ and $\de(3+\ve)$  given by (\ref{delta_def_2}),
(\ref{delta_def_2_ep}) and (\ref{delta_def_3_ep}), respectively.
Then using (\ref{bootstrap_eq_1_general}) and (\ref{improved_1.5}) with (\ref{interp_less}) and letting $r=0$ yields
\begin{align}
 \p \nabla f \p_1 \leq \frac{C}{k^{\frac{5}{2} + \frac{\de_j(1)}{2}}}, 
\label{improved_less}
\end{align}
with $\de(1)$  given by (\ref{delta_def_1}). 

Now one can repeat this argument: given a value of $\de_j(s)$ in 
(\ref{bootstrap_eq_1_general}), estimate (\ref{energy_with_k}) yields 
(\ref{improved_1.5}) with $\de_{j+1}(1.5)$ given by (\ref{delta_def_1.5}); this then gives,
through (\ref{improved_greater}) and (\ref{improved_less}), newly improved values for the remaining  $\de_j(s)$.
One can check from the definitions (\ref{delta_def}) that it is possible 
to iterate this procedure a sufficient number of steps as to obtain
(i) (\ref{bootstrap_improved_eq_1}) with some $\La^\prime$, for $s=1,1.5, 2,2+\ve,3+\ve$; (ii)
$\frac{\de_{j+1}(1.5)}{2} > 1+ \ga$ and $\frac{\de_j(1)}{2} > 1+ \ga$ for some $\ga > 0$.

Now, since $\p \nabla f \p_0 \leq \p \nabla f \p_1$, (\ref{improved_less}) 
gives
\begin{gather}
 \p \nabla f \p_0 \leq \frac{C}{k^{3 + \frac{\ga}{2} } },
\nonumber
\end{gather}
giving the result with $\La = \min\{ \La^\prime, \frac{\ga}{2} \}$.
\hfill $\qed$\\ 

\noindent \emph{Proof for $\p \nabla \dot{f} \p_s$, $s=0,1$.}
Using estimate (\ref{energy_with_k}) and the above,
\begin{gather}
 \p \nabla \dot{f} \p_0 \leq \frac{C k}{k^{\frac{4.5}{2} + \frac{ \de_{j+1}(1.5) }{2} } } 
\leq \frac{C }{k^{\frac{4.5}{2} + \ga } } .
\label{improved_0_dot}
\end{gather}
Using (\ref{improved_0_dot}) with
(\ref{interp_less})
 and letting $r=0$, $s=1$ and $t=4.5$ produces
\begin{gather}
 \p \nabla \dot{f} \p_1 \leq \frac{C}{k^{\frac{3.5}{2} +\frac{3.5}{4.5}\ga} } ,
\nonumber
\end{gather}
giving the result with possibly a new constant $\La$. \hfill $\qed$\\ 

\begin{rema} With (\ref{bootstrap_improved_eq_1}) 
proven for $s=0,1,2,2+\ve,3+\ve$ and with (\ref{bootstrap_improved_eq_2}) proven $s=0,1$, one can apply the interpolation
inequality (\ref{interpolation}) once more in order to obtain
\begin{subequations}
\begin{align}
 & \p \nabla f_k \p_s \leq \frac{\widetilde{C}}{k^{\frac{6-s}{2} + \La} }, ~ s< 6,
\label{bootstrap_improved_more_eq_1} \\
 & \p \nabla \dot{f}_k \p_s \leq \frac{\widetilde{C}}{k^{\frac{4.5-s}{2} + \La}},~ s < 4.5
\label{bootstrap_improved_more_eq_2}
\end{align}
\label{bootstrap_improved_more_eq}
\end{subequations}
\flushleft{possibly} with new constants $\La$ and $\widetilde{C}$. These can be assumed independent of $s$ since we shall use 
(\ref{bootstrap_improved_more_eq}) for only finitely many $s$-values. In particular 
these bounds are valid for integral $s$.
We henceforth assume (\ref{bootstrap_improved_more_eq}).
\end{rema}

\noindent \emph{Proof for $\p \nabla f \p_6$.} From (\ref{algebraic_Df_3_tau})
\begin{gather}
D_\tau^3 f = - 
\frac{r_1 D_\tau^2 D_\nu f + \cA \circ \widetilde{\eta} \, \langle \tau,  W \rangle |D_\tau\widetilde{\eta}|^{-1} + \cR_1}{A}.
\label{D_tau_3_f_improv_6}
\end{gather}
From our constructions and application of (\ref{bilinear}) it is seen that for $\nabla f$ sufficiently small
\begin{align}
\p r_1 \p_{3\half,\partial} +
\p \langle \tau, W \rangle \p_{3\half,\partial} +
 \p \cR_1 \p_{3\half,\partial} \leq C \p D^2 f \p_{3\half,\partial} \leq C \p \nabla f \p_{4\half,\partial}
\leq C \p \nabla f \p_{5}
\label{improv_6_a}
\end{align}
and 
\begin{gather}
 \p | D_\tau \widetilde{\eta} |^{-1} \p_{3\half,\partial} \leq C(1 + \p D^2 f \p_{3\half,\partial} ) 
\leq C(1 + \p \nabla f \p_{5} ) \leq C.
\label{improv_6_b}
\end{gather}
Furthermore, from proposition \ref{prop_estimate_mean_curvature}
with $s = 3$ one obtains
\begin{gather}
 \p \cA \circ \widetilde{\eta} \p_{3\half, \partial} \leq C (1+ \p \nabla
  f \p_{5\half,\partial})
\leq C(1 + \p \nabla f \p_6 ) \leq C,
\label{improv_6_c}
\end{gather}
provided that $\nabla f$ is sufficiently small. We also have
\begin{gather}
 \p D_\tau^2 D_\nu f \p_{3\half,\partial} \leq C \p \nabla f \p_{5\half,\partial} \leq C \p \nabla f \p_6 \leq C.
\label{improv_6_d}
\end{gather}
Using (\ref{improv_6_a}), (\ref{improv_6_b}), (\ref{improv_6_c}) and (\ref{improv_6_d}) along with 
(\ref{bootstrap_improved_more_eq}) and (\ref{D_tau_3_f_improv_6}) yields
\begin{gather}
 \p D_\tau^3 f \p_{3\half, \partial} \leq \frac{C}{k^{\frac{1}{2} + \La}},
\label{estimate_for_bootstrap_D_3_tau_f}
\end{gather}
for some $\La > 0$.

We shall now use equation (\ref{nl_Dirichlet_f_2d}) to estimate $f$. Since this equation is unchanged by subtracting
a constant from $f$, and since our goal is to estimate $\nabla f$, we can assume that
$ \int_{\partial \Om} f = 0$. Then Poincar\'e's inequality implies
\begin{gather}
 \p f \p_{6\half,\partial } \leq C \p D_\tau f \p_{5\half,\partial}.
\nonumber
\end{gather}
But
\begin{gather}
 \p D_\tau f \p_{5\half,\partial} \leq C \left( \p D_\tau f \p_{2,\partial} 
+ \p D^3_\tau f \p_{3\half,\partial} \right) \leq \frac{C}{k^{\frac{1}{2} + \La}} ,
\nonumber
\end{gather}
so that 
\begin{gather}
 \p f \p_{6\half,\partial } \leq \frac{C}{k^{\frac{1}{2} + \La}},
\label{Dirichlet_estimate_6}
\end{gather}
where (\ref{bootstrap_improved_more_eq}) and (\ref{estimate_for_bootstrap_D_3_tau_f}) have been used.
From (\ref{nl_Dirichlet_f_2d}),
\begin{align}
\begin{split}
 \p \nabla f \p_6 & \leq C \p f \p_7 \leq C \Big (
\p \Delta f \p_5 + \p f \p_{6\half,\partial} \Big ) \\
& \leq C \Big (
\p (D^2 f )^2  \p_5 + \p f \p_{6\half,\partial} \Big ).
\end{split}
\label{elliptic_6}
\end{align}
In order to estimate $\p (D^2 f )^2  \p_5 $, let us look at the highest order term
\begin{align}
\begin{split}
 \p D^5 ( D^2 f )^2 \p_0 & \leq C \sum_{\ell=0}^4 \p (D^\ell D^2f) (D^{4-\ell} D^3 f) \p_0  \\
& \leq C \sum_{\ell=0}^4 \p (D^{\ell+1} D f) (D^{6-\ell} D f) \p_0 
\end{split}
\nonumber
\end{align}
Applying (\ref{bilinear}) with $s=1+\ve$, $r=0$ and the term with the $r$-norm on the the right 
of (\ref{bilinear}) being always the one with fewer derivatives in the product
$(D^{\ell+1} D f) (D^{6-\ell} D f)$, leads to
\begin{align}
\begin{split}
 \p D^5 ( D^2 f )^2 \p_0 & \leq C \sum_{\ell=0}^2 \p \nabla f \p_{2+\ell + \ve } \p \nabla f \p_{6-\ell} \\
& \leq \p \nabla f \p_{3+ \ve } \p \nabla f \p_{5}
+ \p \nabla f \p_{4 + \ve } \p \nabla f \p_{4}  + 
\p \nabla f \p_{2+ \ve } \p \nabla f \p_{6} \\
& \leq \frac{C}{k^{\frac{1}{2} + \La}} + \p \nabla f \p_{2+ \ve } \p \nabla f \p_{6}.
\end{split}
\nonumber
\end{align}
The remaining terms in $\p (D^2 f )^2  \p_5 $ are estimated similarly, yielding
\begin{gather}
 \p (D^2 f )^2  \p_5 \leq \frac{C}{k^{\frac{1}{2} + \La}} + \p \nabla f \p_{2+ \ve } \p \nabla f \p_{6}.
\label{D_2_f_sq_estimate}
\end{gather}
Combining 
(\ref{Dirichlet_estimate_6})
(\ref{elliptic_6})
(\ref{D_2_f_sq_estimate}) gives
\begin{gather}
 \p \nabla f \p_6 \leq \frac{C}{k^{\frac{1}{2}+ \La} } + \p \nabla f \p_{2+ \ve } \p \nabla f \p_{6}
\nonumber
\end{gather}
The term $\p \nabla f \p_{2+ \ve } \p \nabla f \p_{6}$ can be absorbed on the left hand side by 
(\ref{bootstrap_improved_more_eq_1}), finishing the proof. \hfill $\qed$ \\

\noindent \emph{Proof for $\p \nabla \dot{f} \p_{4\half}$.} Differentiate 
(\ref{D_tau_3_f_improv_6}) in time and proceed as in the proof of 
$\p \nabla f \p_{6}$.

\section{Proof of theorem \ref{main_theorem}.\label{final_proof}}

We are now in a position to prove theorem \ref{main_theorem}; 
assume therefore all its hypotheses. As we are interested in passing to 
the limit, it will be convenient to reinstate the 
subscript $k$ \footnote{Notice that the operators $L_1$ and $L_2$ also depend
on $k$, but we shall not attach a subscript $k$ to them
as to avoid a cumbersome notation. All we shall need about $L_1$ and $L_2$, in fact,
is that they are bounded operators with norm bounded above by
a constant independent of $k$  --- a condition that clearly
 holds for $k$ large.}.
The $t$-dependence will also be written when we want to stress
the interval where certain estimate is valid. We shall
make  use of proposition \ref{Sobolev_composition}, so 
we remind the reader that $J(\beta_k) = 1$. \\

\noindent \emph{Proof of theorem \ref{main_theorem}:}
For each $k$, let $T_k$ be the maximum interval of existence
for (\ref{free_boundary_full}). Since $\zeta(0) = \id = \eta_k(0) $ and $\dot{\eta}(0) = u_0 = \dot{\eta}_k$,
\begin{align}
\p \zeta(t) - \eta_k(t) \p_{4\half} & = 
\p \zeta(t) - \zeta(0) + \eta_k(0) - \eta_k(t) \p_{4\half} \nonumber
\\
& \leq \p \zeta(t) - \zeta(0) \p_{4\half} + \p \eta_k(t) - \eta_k(0) \p_{4\half},
\nonumber
\end{align}
and
\begin{align}
\p \dot{\zeta}(t) - \dot{\eta}_k(t) \p_{4\half} & = 
\p \dot{\zeta}(t) - \dot{\zeta}(0) + \dot{\eta}_k(0) - \dot{\eta}_k(t) \p_{4\half} \nonumber
\\
& \leq \p \dot{\zeta}(t) - \dot{\zeta}(0) \p_{4\half} + \p \dot{\eta}_k(t) - \dot{\eta}_k(0) \p_{4\half}
\nonumber
\end{align}
As $\zeta$, $\dot{\zeta}$, $\eta_k$ and $\dot{\eta}_k$ are 
continuous functions of $t$,
given $\varrho>0$ we can choose $\overline{T}_k = \overline{T}_k(\varrho)$ 
such that 
\begin{align}
\p \zeta(t) - \eta_k(t) \p_{4\half}  < \varrho, \text{  on } [0,\overline{T}_k],
\label{eta_close_zeta}
\end{align}
and
\begin{align}
\p \dot{\zeta}(t) - \dot{\eta}_k(t) \p_{4\half}  < \varrho, \text{  on } [0,\overline{T}_k].
\label{eta_close_zeta_dot}
\end{align}
Since $\zeta$ exists for all time, there exists a constant $M$ such that 
\begin{gather}
\p \zeta(t) \p_{4\half} + \p \dot{\zeta}_k(t) \p_{4\half}  \leq M, \text{ for } t \leq \overline{T}_k
\label{bound_zeta} 
\end{gather}
Combining (\ref{eta_close_zeta}),
(\ref{eta_close_zeta_dot}) and (\ref{bound_zeta}) yields
\begin{gather}
\p \eta_k(t) \p_{4\half} + \p \dot{\eta}_k(t)  \p_{4\half}
\leq 2M + 2\varrho = M_\varrho,
\text{  on } [0,\overline{T}_k].
\label{bound_eta_eta_dot}
\end{gather}
By proposition \ref{decomposition_more_regular}, the
equations (\ref{bootstrap_eq}) hold on a time interval
$[0,\cT_k]$, and shrinking $\cT_k$ if necessary we can assume $\cT_k \leq \overline{T}_k$. 
Proposition
\ref{prop_energy_bootstrap}
then implies that equations (\ref{bootstrap_improved_eq}) hold
on $\cT_k$. Therefore from (\ref{def_f_beta}) ,  (\ref{bound_eta_eta_dot}) 
and proposition \ref{Sobolev_composition} it follows
\begin{align}
\p \beta_k (t)\p_{4\half} \leq M_\varrho
+ C \p \nabla f_k (t) \p_{4\half}(1 + \p \beta_k  (t) \p_{4\half} ), \text{  on } 
[0,\cT_k].
\nonumber
\end{align}
Because of (\ref{bootstrap_improved_eq_1}), if $k$ is sufficiently large, 
the term 
$ C \p \nabla f_k (t) \p_{4\half} \, \p \beta_k (t) \p_{4\half} $ can be absorbed on the left
hand side and therefore
\begin{align}
\p \beta_k (t)  \p_{4\half} \leq C M_\varrho
+ C \p \nabla f_k  (t) \p_{4\half} \leq C, \text{  on } 
[0,\cT_k],
\label{bound_beta}
\end{align}
where   (\ref{bootstrap_improved_eq_1}) has been used to 
bound $\p \nabla f_k  (t) \p_{4\half}$.
Then, using 
(\ref{bound_beta}) along with proposition \ref{Sobolev_composition} and
 (\ref{beta_inverse_Sobolev}), (\ref{dot_eta_f_beta}), (\ref{eta_eta_dot_bounded}), (\ref{bilinear}), 
 (\ref{bootstrap_improved_eq_2}) gives
\begin{align}
\begin{split}
\p v_k(t) \p_{4\half} & \leq \p \dot{\eta}_k(t)  \circ \beta_k^{-1}(t) \p_{4\half}
 + \p D_{v_k(t)} \nabla f_k(t) \p_{4\half} + \p \nabla \dot{ f}_k(t) \p_{4\half} \\
 & \leq C + \p v_k(t) \p_{4\half} \, \p \nabla f_k (t) \p_{5\half},
 \text{  on } 
[0,\cT_k]. 
  \end{split}
  \nonumber
\end{align}
As before, the last term can be absorbed on the left hand side if $k$ is large 
enough because of (\ref{bootstrap_improved_eq_1}), so that 
\begin{align}
\begin{split}
\p v_k(t) \p_{4\half} & \leq C, \text{  on } 
[0,\cT_k].
  \end{split}
\label{bound_v_45}
\end{align}
From  (\ref{bound_beta}), (\ref{bound_v_45}) and proposition
 \ref{Sobolev_composition} it  immediately follows that 
\begin{align}
\begin{split}
\p \dot{\beta}_k(t) \p_{4\half} & \leq C, \text{  on } 
[0,\cT_k].
  \end{split}
\label{bound_beta_dot}
\end{align}
Therefore, invoking again (\ref{def_f_beta}) we obtain,
with the help of proposition \ref{Sobolev_composition},
(\ref{bound_beta}) and (\ref{bootstrap_improved_eq_1}),
\begin{align}
\p \eta_k(t) -\beta_k(t) \p_{4\half} \leq C \p \nabla f_k(t) \p_{4\half}
\leq \frac{C}{k^\La}, \text{  on } 
[0,\cT_k],
\label{eta_minus_beta}
\end{align}
and from (\ref{dot_eta_f_beta}), (\ref{bound_v_45}), (\ref{bilinear}), 
proposition \ref{Sobolev_composition}, (\ref{bootstrap_improved_eq_1})
and (\ref{bootstrap_improved_eq_2}) we have, after recalling that
$\dot{\beta}_k(t) = v_k(t) \circ \beta_k(t)$,
\begin{align}
\p \dot{\eta}_k(t) -\dot{\beta}_k(t) \p_{4\half} \, \leq C \p \nabla f_k(t) \p_{5\half}
+  C \p \nabla \dot{f}_k(t) \p_{4\half} \leq \frac{C}{k^\La}, \text{  on } 
[0,\cT_k],
\label{eta_minus_beta_dot}
\end{align}
where $\La$ comes from proposition \ref{prop_energy_bootstrap}.
From (\ref{eta_minus_beta}) and (\ref{eta_minus_beta_dot}) we see that in 
order to obtain the desired convergence it will suffice to show that 
$\beta_k(t) \rar \zeta(t)$ and $\dot{\beta}_k(t) \rar \dot{\zeta}(t)$. We proceed
therefore to analyze the equation of motion  for $\beta_k(t)$, i.e., 
(\ref{system_f_v_v_dot}).
Differentiating $\dot{\beta}_k = v_k \circ \beta_k$ in time,
\begin{gather}
\ddot{\beta}_k = ( \dot{v}_k + \nabla_{v_k} v_k ) \circ \beta_k.
\label{beta_v_Lagran}
\end{gather}
But since
\begin{gather}
\dot{v}_k +  P(\nabla_{v_k} v_k ) =
\dot{v}_k +  (I - Q) (\nabla_{v_k} v_k )
\nonumber
\end{gather}
composing (\ref{system_f_v_v_dot}) with $\beta_k$ on the right and 
using (\ref{beta_v_Lagran}) yields
\begin{align}
\ddot{\beta}_k = 
\big (Q(\nabla_{\dot{\beta}_k\circ \beta_k^{-1} }
\dot{\beta}_k\circ \beta_k^{-1}) \big ) \circ \beta_k
- 2 \big ( L_1^{-1}P\nabla_{v_k}
 \nabla\dot{f_k} \big ) \circ \beta_k
  -  \big ( L_1^{-1} P D^2_{v_k v_k} \nabla f_k  \big ) \circ \beta_k .
 \label{eq_beta_prior}
\end{align}
We recall that the pressure $p$ appearing in (\ref{Euler_edo}) is
determined by the fluid velocity $\vartheta = \dot{\zeta} \circ \zeta^{-1}$.
In fact, exactly as in (\ref{beta_v_Lagran}), we have
$\ddot{\zeta} = ( \dot{\vartheta} + \nabla_\vartheta \vartheta )\circ \zeta$,
so that (\ref{Euler_edo}) becomes
\begin{gather}
\dot{\vartheta} + \nabla_\vartheta \vartheta = -\nabla p.
\label{Euler_eq_Eulerian}
\end{gather}
Taking divergence of (\ref{Euler_eq_Eulerian}), using 
$\dive(\dot{\vartheta}) = 0$ and recalling the definition of $Q$, 
(\ref{Euler_edo}) can be written as
\begin{gather}
\ddot{\zeta} =
\big ( Q( \nabla_{\dot{\zeta}\circ \zeta^{-1} } \dot{\zeta}\circ \zeta^{-1} ) \big ) \circ \zeta.
 \label{Euler_Lagran}
\end{gather}
Furthermore, since $\dot{\beta}_k\circ \beta_k^{-1}$
and $\dot{\zeta}\circ \zeta^{-1}$ are on the image of $P$, 
we see that 
(\ref{eq_beta_prior}) and (\ref{Euler_Lagran}) can also be written as
\begin{align}
\ddot{\beta}_k = 
Z(\beta_k, \dot{\beta}_k)
+ \cR_k,
 \label{eq_beta}
\end{align}
and
\begin{gather}
\ddot{\zeta} = Z(\zeta, \dot{\zeta}),
 \label{eq_zeta}
\end{gather}
where $Z$ is the operator defined by
\begin{align}
\begin{split}
& Z: \cD_\mu^{4\half} (\Om) \times H^{4\half}(\Om) \rar H^{4\half}(\Om), \\
& Z(\al, v) = 
\Big( Q \big (  \nabla_{v \circ \al^{-1} } P( v \circ \al^{-1} ) \big ) \Big )\circ \al .
\end{split}
\nonumber
\end{align}
and $\cR_k$ is given by
\begin{gather}
\cR_k(t) = 
- 2 \big ( L_1^{-1}P\nabla_{v_k}
 \nabla\dot{f_k} \big ) \circ \beta_k
  -  \big ( L_1^{-1} P D^2_{v_k v_k} \nabla f_k  \big ) \circ \beta_k .
\nonumber
\end{gather}
$Z$ is well defined since if $w$ is a $H^s$ divergence free vector field, then
the gradient part of $\nabla_w w$, i.e., $Q(\nabla_w w)$ is also in $H^s$ 
and not merely in $H^{s-1}$ as one would expect at first sight. This follows from the
definition of $Q$ since the divergence free condition implies that 
$\dive(\nabla_w w)$ depends on only first derivatives of $w$; see \cite{EM} 
for details.

Integrating equations (\ref{eq_beta}) and (\ref{eq_zeta}), subtracting, 
and using that $\dot{\beta}_k(0) = u_0 = \dot{\zeta}(0)$ 
(because  $\nabla \dot{f}_k(0) = 0$) gives
\begin{align}
\p \dot{\beta}_k(t)  -\dot{\zeta}(t) \p_{4\half} &
\leq \int_0^t \p Z(\beta_k,\dot{\beta}_k)(s) - Z(\zeta, \dot{\zeta})(s) \p_{4\half}
\, ds + \int_0^t \p \cR_k (s) \p_{4\half} \, ds,
\label{beta_minus_zeta_1}
\end{align}
We need to show that $\p \cR_k \p_{4\half}$ is small. For this, first notice that
\begin{gather}
P D_{v_k} \nabla \dot{f}_k = [P, D_{v_k} ] \nabla \dot{f}_k 
+ D_{v_k} P \nabla f = [P, D_{v_k} ] \nabla \dot{f}_k ,
\label{P_nabla}
\end{gather}
where $[\cdot, \cdot ]$ is the
commutator and we used that $P \circ \nabla = 0$.
Recall that if $\cL_1$ and $\cL_2$ are differential operators
of orders $\ell_1$ and $\ell_2$, respectively, then 
their commutator $[\cL_1,\cL_2]$ has order
$\ell_1 + \ell_2 - 1$. Therefore 
$[P,D_{v_k}]$ is a zeroth order operator and hence (\ref{P_nabla})
can be estimated in $H^{4\half}$ with the help of (\ref{bilinear}):
\begin{gather}
\p P D_{v_k} \nabla \dot{f}_k \p_{4\half} \leq C \p v_k \p_{4\half}
\p \nabla \dot{f}_k \p_{4\half}.
\label{P_nabla_bound}
\end{gather}
Analogously, $[P,D_{v_k v_k} ]$ is first order and therefore
\begin{gather}
\p P D_{v_k v_k} \nabla f_k\p_{4\half} \leq C \p v_k \p_{4\half}^2 
\p \nabla f_k \p_{5\half},
\label{P_D_vv_bound}
\end{gather}
where we used again (\ref{bilinear}) and the fact that the operator
$D_{v_k v_k}$ does not depend on derivatives of $v_k$ (see
(\ref{D_vv_two})).  Combining (\ref{P_nabla_bound}) and 
(\ref{P_D_vv_bound}) with  (\ref{bound_beta}), (\ref{bound_v_45}),
(\ref{bootstrap_improved_eq}), and proposition \ref{Sobolev_composition} 
we obtain 
\begin{gather}
\p \cR_k(t) \p_{4\half} \leq \frac{C}{k^\La}, \text{  on }
[0,\cT_k].
\label{estimate_error_Euler}
\end{gather}
Then, from (\ref{estimate_error_Euler}) and  (\ref{beta_minus_zeta_1}), 
it follows
\begin{align}
\p \dot{\beta}_k(t)  -\dot{\zeta}(t) \p_{4\half} &
\leq 
\frac{C}{k^\La} + \int_0^t \p Z(\beta_k,\dot{\beta}_k)(s) - Z(\zeta, \dot{\zeta})(s) \p_{4\half} \, ds, \text{  on } [0,\cT_k].
\label{beta_minus_zeta_2}
\end{align}
In  \cite{E2} it is proven that the map $Z$ is smooth. Since 
$(\zeta(t), \dot{\zeta}(t) )\big( [0,\cT_k] \big )$ is 
compact in $\cD_\mu^{4\half}(\Om) \times H^{4\half}(\Om)$,
the map $Z$ is uniformly Lipschitz in a small neighborhood of the curve
$(\zeta(t), \dot{\zeta}(t) )$. Therefore, because of 
(\ref{eta_close_zeta}), (\ref{eta_close_zeta_dot}), (\ref{eta_minus_beta}) 
and 
(\ref{eta_minus_beta_dot}), if $\varrho$ and $\cT_k$ are sufficiently small
and $k$ sufficiently large, 
\begin{align}
\begin{split}
\p Z(\beta_k,\dot{\beta}_k)(t) - Z(\zeta, \dot{\zeta})(t) \p_{4\half} 
& \leq C \p \beta_k(t) - \zeta(t) \p_{4\half} + 
C \p \dot{\beta}_k(t) - \dot{\zeta}(t) \p_{4\half} 
\text{  on } [0,\cT_k].
\end{split}
\label{Lip_estimate}
\end{align}
But,
\begin{align}
 \p \beta_k(t) - \zeta(t) \p_{4\half} & \leq 
 \int_0^t  \p \dot{\beta}_k(s) - \dot{\zeta}(s) \p_{4\half} \, ds.
 \label{beta_zeta_integral}
\end{align}
Using (\ref{Lip_estimate}) and  (\ref{beta_zeta_integral}) 
into (\ref{beta_minus_zeta_2}) we obtain, after iterating the
 inequality (\ref{beta_minus_zeta_2}), that
 \begin{align}
 \p \beta_k(t) - \zeta(t) \p_{4\half}  +  \p
 \dot{\beta}_k(t) - \dot{\zeta}(t) \p_{4\half}
& \leq \frac{C}{k^\La} e^{C t}, \text{   on } [0,\cT_k].
\label{estimate_distance_beta_zeta}
\end{align}
Recall now that the time interval was defined by the validity of 
(\ref{bootstrap_eq}). Proposition 
\ref{prop_energy_bootstrap} shows that this implies that a better estimate, namely,
(\ref{bootstrap_improved_eq}), holds on $[0,\cT_k]$, and thus  
if $k$ is sufficiently large, $\cT_k$ is a non-decreasing function 
of $k$, as long the solution $\eta_k(t)$ exists on that interval.
But now, since the solution $\zeta$ exists for all time, it follows from
(\ref{estimate_distance_beta_zeta}), equations (\ref{eq_zeta}), 
(\ref{eq_beta}) and (\ref{estimate_error_Euler}), that the interval of existence 
of $\beta_k(t)$ is also non-decreasing with $k$ (for $k$ large), and then 
the same is true for $\eta_k(t)$ by (\ref{eta_minus_beta}) and (\ref{eta_minus_beta_dot}). We conclude that there exists  a $T>0$, independent
of $k$, such that $\eta_k(t)$ exists on $[0,T]$  for all $k$,
and such that (after combining (\ref{estimate_distance_beta_zeta}) with
(\ref{eta_minus_beta}) and (\ref{eta_minus_beta_dot})),
\begin{align}
 \sup_{0 \leq t \leq T} \p \eta_k(t) - \zeta(t) \p_{4\half}  
 +  \sup_{0 \leq t \leq T} \p  \dot{\eta}_k(t) - \dot{\zeta}(t) \p_{4\half}
\, \longrightarrow  0 \text{  as } k \rar \infty,
\nonumber
\end{align}
finishing the proof. \hfill $\qed$

\subsection{Final remarks.\label{final_remarks}}
Here we expand on several comments made in the introduction,
and make some further remarks.

We start by pointing out how the theorem can be generalized to include 
different initial conditions $u_{0 k}$, not necessarily tangent to the boundary, 
but which converge suitably to an initial velocity $u_0$ of problem
(\ref{Euler}). In this case, we would have
\begin{gather}
u_{0k} = P(u_{0k}) + Q(u_{0k}) 
= v_{0k} + \nabla \dot{f}_{0k}.
\nonumber
\end{gather}
We then have to choose
\begin{gather}
v_{0 k } - u_0 \sim \frac{1}{k^A} \text{ \hskip 0.3cm and \hskip 0.3cm }
\nabla \dot{f}_{0k}  \sim \frac{1}{k^A},
\nonumber
\end{gather}
in the appropriate norms. It is not difficult to see then that
if $A$ is large, the estimates we presented will still go through. In fact, in the 
energy estimates, there will be an additional term $\widetilde{u}_{0k}$
coming from $\nabla \dot{f}_k(0)$, which can then be grouped
with the other terms in $\frac{1}{k}$ without affecting the result, 
provided that $A$ is suitably chosen for this purpose.
Similarly, when we considered 
$\p \dot{\zeta}(t) - \dot{\eta}_k(t) \p_{4\half}$
and $\p \dot{\zeta}(t) - \dot{\beta}_k(t) \p_{4\half}$ in the previous proof, 
we would have extra terms
of the form $\p u_0 - ( v_{0k} + \nabla \dot{f}_{0k} ) \p_{4\half}$
and $\p u_0 -  v_{0k} \p_{4\half}$, respectively, which again will be negligible errors
of the order $\frac{1}{k^A}$.

The importance of having only one boundary component with constant mean
curvature at time zero can be seen from 
the boundary condition for the pressure $p$. $p$ is explicitly coupled to $\eta$ via 
its gradient. The ``boundary part'' of $p$ is given by the harmonic extension $\cA_H$ of the 
the mean curvature $\cA$, and therefore $\nabla \cA_H = 0$ when $\cA$ is constant so one can expect $\nabla \cA_H$
to be small for small time. 

We point out that our method can also be used to prove well-posedness in
$\ccE_\mu^s(\Om)$. This is done by using an iteration scheme for the system
(\ref{system_f_v_full}). The iteration will converge if $f$ is sufficiently small,
which, due to the 
estimates of section \ref{space_smoother_emb}, will be the case provided that $\de_0$ in theorem \ref{embedding} 
is small enough. This will be shown in a future work \cite{in_prep}.

We finally comment on the higher dimensional case. There seems to be no conceptual 
problem in generalizing our proofs to three and higher dimensions, the difficulties being
mostly computational. Notice that in section \ref{final_proof} we did use the fact that solutions to 
(\ref{Euler}) exist for all time, a feature unique to those equations in two spatial dimensions, 
but this is superfluous. All we need there is a fixed time interval $[0,T_*]$ in which
(\ref{Euler}) has a unique solution (see \cite{E2, E4, ED} where a similar 
convergence argument is used). In fact, we expect that most of the 
arguments here presented
can be generalized to higher dimensions in a more or less straightforward manner. The part
that presents substantial difficulties is the derivation of the energy estimates of section
\ref{energy_esimates_section}, in particular those dealing with the mean curvature of the boundary.
The simplicity of the mean curvature operator for one-dimensional boundaries was employed
to a great extent in section \ref{energy_esimates_section}. The expressions will become much 
harder to handle in three and higher dimensions, as it should be expected from usual
difficulties posed by equations involving mean curvature.

\end{document}